\documentclass[EJP]{ejpecp} 

\usepackage[utf8]{inputenc}

\usepackage[square]{natbib}
\usepackage{titlesec}
\usepackage{parskip}
\usepackage{listings} 
\usepackage{lineno} 
\usepackage[normalem]{ulem} 
\usepackage{xcolor}
\usepackage{tabularx}
\usepackage{stmaryrd}
\usepackage{enumerate}
\usepackage{mathrsfs}
\usepackage{subcaption}
\usepackage{float}
\usepackage{booktabs}
\usepackage{authblk}
\usepackage{enumerate}  
\usepackage{orcidlink}
\usepackage{placeins}

\SHORTTITLE{Optimising two-block averaging kernels}

\TITLE{Optimising two-block averaging kernels to speed up Markov chains} 

\AUTHORS{%
  Ryan~J.Y.~Lim\footnote{Department of Statistics and Data Science, National University of Singapore,
    Level 7, 6 Science Drive 2, 117546, Singapore.}%
  \footnote{\EMAIL{ryan.limjy@u.nus.edu}}%
  \and
  Michael~C.H.~Choi\textsuperscript{*}%
  \footnote{\EMAIL{mchchoi@nus.edu.sg}, corresponding author}\orcid{0000-0003-0309-3217}%
}


\AMSSUBJ{60J10, 60J22, 65C40, 90C27, 94A15, 94A17} 

\SUBMITTED{TO FILL} 
\ACCEPTED{} 

\VOLUME{0}
\YEAR{2023}
\PAPERNUM{0}
\DOI{10.1214/YY-TN}

\ABSTRACT{We study the problem of selecting optimal two-block partitions to accelerate the mixing of finite Markov chains under group-averaging transformations. The main objectives considered are the Kullback-Leibler (KL) divergence and the Frobenius distance to stationarity. We establish explicit connections between these objectives and the induced projection chain. In the case of the KL divergence, this reduction yields explicit decay rates in terms of the log-Sobolev constant. For the Frobenius distance, we identify a Cheeger-type functional that characterises optimal cuts. This formulation recasts two-block selection as a structured combinatorial optimisation problem admitting difference-of-submodular decompositions. We further propose several algorithmic approximations, including majorisation–minimisation and coordinate descent schemes, as computationally feasible alternatives to exhaustive combinatorial search. Our numerical experiments reveal that optimal cuts under the two objectives can substantially reduce total variation distance to stationarity and demonstrate the practical effectiveness of the proposed approximation algorithms.}

\DeclareMathOperator{\Tr}{Tr}
\DeclareMathOperator*{\argmin}{arg\,min}
\DeclareMathOperator*{\argmax}{arg\,max}

\begin{document}



\section{Introduction}
This work is motivated by recent advancements in group-averaged Markov chains, where in \cite{GPG}, it is shown that given a Gibbs kernel $G$, group-averaging kernels of the form $GPG$ can exhibit significant theoretical improvements over the base kernel $P$. Improvements in metrics such as mixing time, spectral gaps and Kullback-Leibler (KL) divergence have been shown. However, while the mechanism of group-averaging was thoroughly analysed, the question of how to choose the underlying group or partition remained largely open. In particular, beyond the presence of an intrinsic symmetry, it was unclear how one should select blocks in order to maximise improvements. In this paper, we investigate this optimisation problem, focusing primarily on the special case of two blocks. 

Precisely, suppose that there is a base $\pi$-stationary Markov chain sampler with transition matrix $P$ on a finite state space $\mathcal{X} = S \sqcup S'$, where $S \subset \mathcal{X}$ and $S'$ is the complement of $S$ with respect to $\mathcal{X}$. Let $G = G_S$ denote the Gibbs kernel with respect to the partition $\mathcal{X} = S \sqcup S'$. We shall be interested in (approximately) solving combinatorial optimisation problems of the form
\begin{align*}
    &\min_{S \subset \mathcal{X};~ S \neq \emptyset, \mathcal{X}} \mathrm{dist}(G_SP, \Pi), \quad \min_{S \subset \mathcal{X};~ S \neq \emptyset, \mathcal{X}} \mathrm{dist}(P G_S, \Pi), \\
    &\min_{S \subset \mathcal{X};~ S \neq \emptyset, \mathcal{X}} \mathrm{dist}(G_SPG_S, \Pi), \quad \min_{S,V \subset \mathcal{X};~ S,V \neq \emptyset, \mathcal{X}} \mathrm{dist}(G_VPG_S, \Pi),
\end{align*}
where $\Pi$ is the transition matrix whose rows equal to $\pi$ respectively, and $\mathrm{dist}$ is either the KL divergence or the (squared)-Frobenius norm. The above notions shall be recalled properly in Section \ref{sec:prelim}, but we briefly present them here in order to motivate the setup. Solving these optimisation problems offers a way and gives insight on optimally choosing a two-block partition that minimises a suitable ``distance" to stationarity. Interestingly, we show that these problems are closely connected to Cheeger-type functionals, entropy, and the projection chain $\overline{P}$ of $P$ with respect to the partition $(S,S')$. Moreover, they admit decompositions into differences of submodular functions, allowing tools from submodular optimisation to assist in identifying optimal or near-optimal partitions.

Solving the above combinatorial optimisation problems also has important consequences in mixing time analysis. Specifically, we shall prove in Proposition \ref{prop:GSPGSlPi} below that for $l \in \mathbb{N}$ and $\pi$-reversible $P$ we have
\begin{align*}
    \argmin_{S \neq \mathcal{X}, \emptyset} \|(G_SPG_S)^l - \Pi\|_{F, \pi}^2 = \argmin_{S \neq \mathcal{X}, \emptyset} \|G_SPG_S - \Pi\|_{F, \pi}^2,
\end{align*}
where $\|\cdot\|_{F, \pi}$ is the Frobenius norm. Therefore, to minimise the long-time discrepancy with stationarity in squared-Frobenius norm, it suffices to find a minimiser in one-step distance to stationarity rather than investigating the $l$-step distance to stationarity. 

The rest of this paper is organized as follows. We first establish in Section \ref{sec:KL} that the KL divergence from stationarity of the group-averaged kernels $(GPG)^l$ is equal to that of the projection chain $(\overline{P})^l$ and its stationary distribution. In the case where $G$ consists of two blocks, this reduction allows us to derive an explicit decay rate in terms of the log-Sobolev constant of $\overline{P}^2$. Moreover, via data-processing inequalities, the same decay rate extends to the kernels $GP$ and $PG$.

In Section \ref{sec:frob}, we turn to the Frobenius distance to stationarity for the kernels $GPG$, $PG$ and $GP$. In the two-orbit case, we show that the optimisation problem reduces to maximising a Cheeger-type functional $g(S).$ Interestingly, the symmetrised Cheeger's cut turns out to be the worst choice for minimising the Frobenius objective. We further obtain a $1/2$-approximation result, which reduces the original combinatorial optimisation problem to a simple linear sweep over singleton sets.

Section \ref{sec:frobeg} extends this perspective to the multi-orbit setting. We prove that orbit-averaging reduces the Frobenius distance to stationarity to order $\mathcal{O}(k)$, where $k$ is the number of orbits. This is a substantial improvement over lazy kernels $P$, whose Frobenius distance is always of order $\Omega(n)$, where $n$ is the size of $\mathcal{X}$. A hypercube example demonstrates that in the case $k=2$, using even the Cheeger's cut, which maximises (rather than minimises) the Frobenius objective can yield significant improvements over the lazy random walk sampler.

In Section \ref{sec:opt}, we develop practical algorithmic approximations by decomposing both the KL divergence and Frobenius objectives into differences of submodular functions in the two-orbit case. This structure enables the use of established submodular optimisation procedures to approximate the underlying combinatorial problems. In particular, the KL objective $ D_{\mathrm{KL}}^\pi(PG_S \| \Pi)$ admits a decomposition directly linked to the entropy of 
$\pi$ and the entropy rate of the kernel $PG_S.$

Section \ref{sec:posdef} investigates structural properties of positive definite kernels with group-averaging. We show that, for any positive definite, $\pi$-reversible $P$, if neither $P$ nor $G$ is exactly $\Pi$, it is impossible for $GPG$ or $GP$ to be $\Pi.$

Finally, Section \ref{sec:num} presents numerical experiments using the Curie–Weiss model with Glauber dynamics as a testbed. Figures \ref{fig:tvplot} and \ref{fig:tvplotGS} demonstrate that group-averaging significantly outperforms the base kernel $P$ in total variation distance, even when the partition is chosen randomly. The experiments further suggests that the proposed approximation schemes perform particularly well in skewed energy landscapes, where the objective is strongly structured and majorisation–minimisation (MM) procedures exhibit reliable descent behaviour.

\subsection{Related works}
A substantial body of work has investigated structural modifications of Markov chains to improve convergence. Classical approaches include lifting, where \cite{chen_lifting} propose enlarging the transition structure in order to accelerate mixing. More recently, systematic constructions of non-reversible samplers have been developed in both discrete and continuous settings; see, for instance, \cite{TURITSYN2011410} and \cite{bierkens_2016}. These works demonstrate that altering the structure of the transition kernel can yield substantial improvements in mixing behaviour. In particular, the proposed samplers $G_SP$ and $PG_S$ may also be viewed as non-reversible transformations of the base kernel $P$.

Beyond structural modifications, a complementary line of research focuses on tuning sampling algorithms by optimising their parameters. Adaptive MCMC methods, first proposed in \cite{haario_adaptive}, adjust proposal distributions during simulation. \cite{atchade_2005} subsequently proposed adaptive schemes that sequentially learn near-optimal parameters. From an optimisation perspective, \cite{chen_2012} and \cite{huang_2018} study the construction of ``optimal" transition kernels that minimise criteria based upon asymptotic variance and its variants. More recently, \cite{pmlr-v235-egorov24a} introduced adversarial and learning-based approaches for training transition kernels with improved mixing performance. In these settings, aspects of the sampling mechanism are treated explicitly as optimisation variables.

Group-averaging is likewise a structural modification designed to improve mixing. Building on prior work of \cite{choi2025groupaveragedmarkovchainsmixing} and \cite{GPG}, we study the optimisation problem that arises when the underlying partition is not prescribed by symmetry or model structure. In the special case of two blocks, this reduces to selecting optimal cuts under various performance metrics. By relating the resulting combinatorial optimisation problem to the induced projection chain, we develop tractable approximation algorithms aimed at circumventing the inherent difficulty of exhaustive combinatorial search. This perspective places group-averaging within the broader framework of optimisation-driven sampler design.

\section{Preliminaries and definitions}\label{sec:prelim}
Throughout this paper, we define $\mathcal{X}$ to be a finite state space, and let $\mathcal{P}(\mathcal{X})$ denote the set of all probability measures on $\mathcal{X}$ with strictly positive masses, i.e., $\mathcal{P}(\mathcal{X}) := \{ \pi : \mathcal{X} \to (0,1] : \sum_x \pi(x)=1 \}$.

For integers $a \leq b \in \mathbb{Z}$, we write $\llbracket a,b \rrbracket := \{a, a+1, \ldots, b\}$ and $\llbracket n \rrbracket := \llbracket 1,n \rrbracket$ with $n \in \mathbb{N}$. From here on, we always assume that $\mathcal{X} = \llbracket n \rrbracket$ unless otherwise specified. 

Let $f,g:\mathbb{N}\to\mathbb{R}^+$. We write $f(n)=\mathcal{O}(g(n))$ if there exist constants $c>0$ and $n_0\in\mathbb{N}$ such that $f(n)\le c\,g(n)$ for all $n\ge n_0$. We write $f(n)=\Omega(g(n))$ if there exist constants $c>0$ and $n_0\in\mathbb{N}$ such that $f(n)\ge c\,g(n)$ for all $n\ge n_0$.

Let $\ell^2(\pi)$ be the Hilbert space weighted by $\pi \in \mathcal{P}(\mathcal{X})$, with inner product defined to be 
$$\langle f,g \rangle_\pi := \sum_{x\in \mathcal{X}} f(x)g(x)\pi(x),$$
for $f,g: \mathcal{X} \to \mathbb{R}$. The $\ell^2(\pi)$-norm of $f$ is then $\|f\|_\pi = \sqrt{\langle f,f \rangle_\pi}$. We write 
$$\ell^2_0(\pi) := \{f \in \ell^2(\pi): \pi(f) = 0\}$$
as the zero-mean subspace.

Let $\mathcal{L} = \mathcal{L}(\mathcal{X})$ to be the set of all transition matrices on $\mathcal{X}$. For a given $\pi \in \mathcal{P}(\mathcal{X})$, the set of all $\pi$-stationary transition matrices shall be defined as $\mathcal{S}(\pi) \subseteq \mathcal{L}$. That is, for any $P \in \mathcal{S}(\pi),$ it must satisfy $\pi P = \pi$. Similarly, the set of all $\pi$-reversible matrices is defined as $\mathcal{L}(\pi) \subseteq \mathcal{L}$, where $P \in \mathcal{L}(\pi)$ implies that the detailed balance condition is satisfied for $P$. That is, $\pi(x) P(x,y) = \pi(y) P(y,x)$ holds for all $x,y \in \mathcal{X}$. 

For any $M \in \mathbb{R}^{n \times n}$, we define $M^*$ to be the $\ell^2(\pi)$-adjoint of $M$. In particular, for ergodic $P\in \mathcal{S}(\pi)$, $P^*$ is the time-reversal of $P$, and we have $P\in \mathcal{L}(\pi)$ if and only if $P^* = P$.

The transition matrices $P \in \mathcal{L}$ can also be viewed as an operator on $\ell^2(\pi)$. Then 
$$Pf(x) = \sum_{y \in \mathcal{X}} P(x,y) f(y)$$
is also a function in $\ell^2(\pi)$.

For any $M, N \in \mathbb{R}^{n \times n}$, we define the Frobenius inner product of $M, N$ with respect to $\pi$ to be 
$$\langle M,N \rangle_{F,\pi} := \Tr(M^*N),$$
and the induced Frobenius norm $\| M\|_{F, \pi}^2 := \langle M,M\rangle_{F,\pi}$, where $\Tr(M)$ is the trace of $M$. For clarity, we prove below that $\langle \cdot,\cdot \rangle_{F,\pi}$ defined above is indeed an inner product.

    Let $M, N, Q \in \mathbb{R}^{n \times n}$, and $a,b \in \mathbb{R}$.
    The map $\langle \cdot,\cdot \rangle_{F,\pi}: \mathbb{R}^{n \times n} \times \mathbb{R}^{n \times n} \to \mathbb{R}$ satisfies the following properties:
    \begin{itemize}
        \item Symmetry: 
        $$\langle M, N \rangle_{F, \pi} = \Tr(M^*N) = \Tr(N^*M) = \langle N,M \rangle_{F, \pi}.$$

        \item Linearity: 
        \begin{align*}
            \langle aM+bN, Q \rangle_{F, \pi} &= \Tr((aM+bN)^* Q)\\
            &= \Tr(aM^*Q) + \Tr(bN^*Q)\\
            &= a\langle M, Q \rangle_{F, \pi} + b\langle N, Q \rangle_{F, \pi}.
        \end{align*}

        \item Positive-definiteness: For $M \neq \mathbf{0}_{n \times n}$, note that 
        $$(M^*M)(x,x) = \sum_{z\in \mathcal{X}} M^*(x,z) M(z,x) = \sum_{z\in \mathcal{X}} \frac{\pi(z)}{\pi(x)} M(z,x)^2 \geq 0.$$
        It follows that 
        $$\langle M,M \rangle_{F, \pi} = \Tr(M^*M) \geq 0,$$
        and equality holds if and only if $M = \mathbf{0}_{n \times n}.$
    \end{itemize}

\subsection{Eigenvalues and Cheeger's constant}

In subsequent sections, we shall link the Frobenius norm to stationarity of $P$ to its eigenvalues and Cheeger's constant. Thus, we briefly recall these notions in this subsection.

For self-adjoint $M \in \mathbb{R}^{n \times n}$, we write its eigenvalues $(\lambda_i = \lambda_i(M))_{i=1}^n$ in non-increasing order counted with multiplicities, that is,
\begin{align*}
    \lambda_1(M) \geq \lambda_2(M) \geq \ldots \geq \lambda_n(M).
\end{align*}
For $P \in \mathcal{L}(\pi)$, we write $\gamma(P) := 1 - \lambda_2(P)$ to be the right spectral gap of $P$.

The classical Cheeger's constant of $P$ is defined to be
\begin{align*}
    \Phi^*(P) := \min_{S;~0 < \pi(S) \leq \frac{1}{2}} \dfrac{\sum_{x \in S;\,y \in S'} \pi(x)P(x,y)}{\pi(S)},
\end{align*}
and we call $S^* \in \argmin_{S;~0 < \pi(S) \leq \frac{1}{2}} \frac{\sum_{x \in S;\,y \in S'} \pi(x)P(x,y)}{\pi(S)}$ a Cheeger's cut of $P$, where we write $S'$ to be the complement of the set $S$. Similarly, the symmetrised Cheeger's constant of $P$ introduced in Section 4.1 of \cite{Montenegro} is defined to be
\begin{align*}
    \phi^*(P) := \min_{S;~0 < \pi(S) < 1} \dfrac{\sum_{x \in S;\,y \in S'} \pi(x)P(x,y)}{\pi(S)\pi(S')},
\end{align*}
and we call $U^* \in \argmin_{S;~0 < \pi(S) < 1} \frac{\sum_{x \in S;\,y \in S'} \pi(x)P(x,y)}{\pi(S)\pi(S')}$ a symmetrised Cheeger's cut of $P$. 

The classical Cheeger's inequality (see \cite{roch_mdp_2024} Theorem 5.3.5 for a proof) gives, for ergodic $P \in \mathcal{L}(\pi)$, that
\begin{align}\label{eq:Cheegerineq}
    \dfrac{\Phi^{*2}(P)}{2} \leq \gamma(P) \leq 2 \Phi^*(P).
\end{align}
It is also immediate to see that
\begin{align*}
    \phi^*(P) \leq 2 \Phi^*(P).
\end{align*}

\subsection{Submodular functions, KL divergence and entropy}

For any $P, Q \in \mathcal{S}(\pi)$, the KL divergence of $P$ from $Q$ weighted by $\pi$ is defined as
\begin{equation*}
    D_{\mathrm{KL}}^\pi (P \| Q) := \sum_{x,y \in \mathcal{X}} \pi(x) P(x,y) \log\bigg(\frac{P(x,y)}{Q(x,y)}\bigg),
\end{equation*}
where by convention we take $0 \log (0/a) := 0$ for $a \in [0,1]$. For any probability distribution $\mu, \nu$ in $\mathcal{P}(\mathcal{X})$, the KL divergence of $\mu$ from $\nu$ is defined to be
\begin{align*}
    D_{\mathrm{KL}}(\mu \| \nu) := \sum_{x \in \mathcal{X}} \mu(x) \log \left(\dfrac{\mu(x)}{\nu(x)}\right).
\end{align*}

Further, for any distribution $\pi$ on $\mathcal{X}$, let 
$$H(\pi) := -\sum_{x\in \mathcal{X}} \pi(x) \log \pi(x)$$
be the Shannon entropy of $\pi$ on $\mathcal{X}$.

Similarly, we write 
$$H_\pi(P) := -\sum_{x, y \in \mathcal{X}} \pi(x)P(x,y)\log P(x,y)$$
to be the entropy rate of any ergodic, $\pi$-stationary $P \in \mathcal{S}(\pi)$.

We now give a brief introduction of sub/supermodular functions. Let $f: 2^V \to \mathbb{R}$ be a set function where $f(S)$ is real-valued for any $S \subseteq V.$ We say a set function $f$ is submodular if for every $A, B \subseteq V,$
$$f(A \cap B) + f(A \cup B) \leq f(A) + f(B).$$ A function $f$ is called supermodular if $-f$ is submodular. 

If equality holds for all $A, B$, then the function $f$ is modular. A modular function is thus both supermodular and submodular. Other equivalent definitions of these functions can be found in \cite{Krause_Golovin_2014}.

\subsection{Gibbs kernel}
Consider a partition of the state space $\mathcal{X} = \bigsqcup_{i=1}^k \mathcal{O}_i$, where $(\mathcal{O}_i)_{i=1}^k$ is the induced orbits obtained from an appropriate choice of group action acting on $\mathcal{X}$. Note that any partition of $\mathcal{X}$ can in fact be induced by a product of cyclic groups on the sets $\mathcal{O}_i$.

From \cite{GPG}, we define the Gibbs kernel $G \in \mathcal{L}(\pi)$ to be 
\begin{equation*} 
G(x,y) = \begin{cases}
\frac{\pi(y)}{\pi(\mathcal{O}(x))},& \mathrm{for}\ y \in \mathcal{O}(x),\\
0, & \mathrm{otherwise},
\end{cases}
\end{equation*}
where $\pi(\mathcal{O}(x)) := \sum_{z \in \mathcal{O}(x)} \pi(z)$ and $\mathcal{O}(x)$ is defined to be the orbit that $x$ belongs, that is, $\mathcal{O}(x) = \mathcal{O}_i$ with $i$ satisfying $x \in \mathcal{O}_i$. Note that $\mathcal{O}(x)$ should not be confused with the big O notation.

From its definition, $G$ is $\pi$-reversible and idempotent, that is, $G^2 = G.$ Intuitively, the kernel $G$ corresponds to resampling from the target distribution $\pi$ restricted to the orbit $\mathcal{O}(x)$ when the initial state is $x$. 

The main motivation behind introducing $G$ is that it facilitates within orbit transition, which may be difficult for a given $\pi$-stationary sampler $P$ to achieve. In particular, appropriate choices of $G$ can help samplers traverse between modes more efficiently, avoiding situations where a sampler gets stuck in a single region of the state space.

\section{Relationship between $GPG$ and $\overline{P}$ in KL divergence} \label{sec:KL}

In this section, we will show that the KL divergence to stationarity for $GPG$ depends only on the projection chain $\overline{P}$ for any choice of partition $\mathcal{X} = \bigsqcup_{i=1}^k \mathcal{O}_i.$

Recall from \cite{Jerrum_2004}, the projection chain $\overline{P}$ induced by the partition $(\mathcal{O}_i)_{i=1}^k$, with $\overline{P}: \llbracket k \rrbracket \times \llbracket k \rrbracket \to [0,1]$ is defined to be

\begin{align} \label{eq:Pbar}
    \overline{P}(i,j) := \frac{1}{\pi(\mathcal{O}_i)} \sum_{\substack{x \in \mathcal{O}_i \\ y \in \mathcal{O}_j}} \pi(x) P(x,y).
\end{align}
It follows that $\overline{P}$ has stationary distribution $\overline{\pi} = (\pi(\mathcal{O}_1), \dots, \pi(\mathcal{O}_k)).$ Note that we similarly define $\overline{\Pi}$ to be the transition matrix whose rows are equals to $\overline{\pi}$.

\begin{proposition}\label{prop:GPGoverlinePrs}
    For any $P \in \mathcal{S}(\pi)$ and any partition $\mathcal{X} = \bigsqcup_{i=1}^k \mathcal{O}_i$, it holds that for $l \in \mathbb{N}$
    \begin{equation*}
        (GPG)^l (x,y) = \overline{P}^l (i,j) \frac{\pi(y)}{\pi(\mathcal{O}_j)},
    \end{equation*}
    where $x \in \mathcal{O}_i$, $y \in \mathcal{O}_j.$
\end{proposition}

\begin{proof}
    For $l = 1$, when  $x \in \mathcal{O}_i$, $y \in \mathcal{O}_j,$
    $$GPG(x,y) = \frac{\pi(y)}{\pi(\mathcal{O}_i)\pi(\mathcal{O}_j)} \sum_{\substack{z \in \mathcal{O}_i\\ w \in \mathcal{O}_j}} \pi(z) P(z,w) = \overline{P} (i,j) \frac{\pi(y)}{\pi(\mathcal{O}_j)}.$$
    Suppose then that the claim holds for all $m \leq  l$, where $l \geq 1$ is fixed. 
    It follows that 
    \begin{align*}
        (GPG)^{l+1}(x,y) &= \sum_{z \in \mathcal{X}} (GPG)^l(x,z)\cdot GPG(z,y)\\
        &= \sum_{h=1}^k \sum_{z \in \mathcal{O}_h} \overline{P}^l (i,h) \frac{\pi(z)}{\pi(\mathcal{O}_h)} \cdot \overline{P}(h,j) \frac{\pi(y)}{\pi(\mathcal{O}_j)}\\
        &= \sum_{h=1}^k \overline{P}^l (i,h)\cdot \overline{P}(h,j) \sum_{z \in \mathcal{O}_h} \frac{\pi(y)\pi(z)}{\pi(\mathcal{O}_h)\pi(\mathcal{O}_j)}\\
        &= \overline{P}^{l+1}(i,j) \frac{\pi(y)}{\pi(\mathcal{O}_j)}.
    \end{align*}
    By induction, we have the result as claimed.
\end{proof}

\begin{proposition} \label{prop:GPGoverlineP KL}
    For any $P \in \mathcal{S}(\pi)$, $\mathcal{X} = \bigsqcup_{i=1}^k \mathcal{O}_i$ and $l \in \mathbb{N}$, we have
    $$D_{\mathrm{KL}}^\pi \left((GPG)^l \| \Pi \right) = D_{\mathrm{KL}}^{\overline\pi} \left( \overline{P}^l \| \overline{\Pi}\right).$$
\end{proposition}

\begin{proof}
    By Proposition \ref{prop:GPGoverlinePrs}, 
    \begin{align*}
        D_{\mathrm{KL}}^\pi \left((GPG)^l \| \Pi \right) &= \sum_{x,y \in \mathcal{X}} \pi(x) (GPG)^l(x,y) \log \frac{(GPG)^l(x,y)}{\pi(y)}\\
        &= \sum_{i,j=1}^k \sum_{\substack{x \in \mathcal{O}_i\\ y \in \mathcal{O}_j}} \pi(x) \overline{P}^l(i,j) \frac{\pi(y)}{\pi(\mathcal{O}_j)} \log \frac{\overline{P}^l(i,j)}{\pi(\mathcal{O}_j)}\\
        &= \sum_{i,j=1}^k \pi(\mathcal{O}_i) \overline{P}^l(i,j) \log \frac{\overline{P}^l(i,j)}{\pi(\mathcal{O}_j)}\\
        &= D_{\mathrm{KL}}^{\overline\pi} \left( \overline{P}^l \| \overline{\Pi}\right).
    \end{align*}
\end{proof}
We remark that the above proposition is a generalisation of Proposition 6.3 of \cite{GPG}, where the case $l=1$ was presented.

The results of Proposition \ref{prop:GPGoverlineP KL} motivates the study of the projection chain $\overline{P}$ over the full kernel $GPG$. In particular, the convergence rate of $(GPG)^l$ to stationarity in KL divergence is exactly that of $\overline{P}^l$ to $\overline{\Pi}$. When we partition the state space into exactly two blocks, $\overline{P}$ is then a $2 \times 2$ matrix which makes it easy to analyse its properties.

\subsection{Decay rate of the KL divergence of $(GPG)^l$ from $\Pi$ and the case of $k=2$ orbits}

In this subsection, we bound the decay rate of the KL divergence from $\Pi$ to $(GPG)^l$ in terms of the log-Sobolev constant of $\overline{P}^2$.

First, we define the log-Sobolev constant of a sampler $P \in \mathcal{S}(\pi)$ to be
\begin{equation} \label{eq:alpha}
    \alpha(P) := \min \left\{ \frac{\langle (I - P)f, f \rangle_\pi}{\mathscr{L}(f)} : \mathscr{L}(f) \neq 0 \right\},
\end{equation}
where the quantity $\mathscr{L}(f)$ is defined to be 
$$\mathscr{L}(f) := \sum_{x\in \mathcal{X}} |f(x)|^2 \log\left(\frac{|f(x)|^2}{\|f\|_\pi^2}  \right)\pi(x).$$
\begin{proposition} \label{prop:logsobolevbound}
    For $P \in \mathcal{L}(\pi)$, $\mathcal{X} = \bigsqcup_{i=1}^k \mathcal{O}_i$ and $l \in \mathbb{N}$, we have
    $$D_{\mathrm{KL}}^\pi \left((GPG)^l \| \Pi \right) = D_{\mathrm{KL}}^{\overline\pi} \left( \overline{P}^l \| \overline{\Pi}\right) \leq \left(1-\alpha\left(\overline{P}^2\right)\right)^l H(\overline{\pi}).$$
\end{proposition}

\begin{proof}
    The equality comes from Proposition \ref{prop:GPGoverlineP KL}. Note that for $P \in \mathcal{L}(\pi)$, $\overline{P}$ must be $\overline{\pi}$-reversible, and hence we have $$\alpha\left(\overline{P}\cdot \overline{P}^*\right) = \alpha\left(\overline{P}^2\right).$$
    Denote $\delta_i$ to be the Dirac mass at $i \in \llbracket k \rrbracket$. Note that we have
    \begin{align*}
        D_{\mathrm{KL}}(\delta_i \overline{P}^l \| \overline{\pi}) \leq \left(1-\alpha\left(\overline{P}^2\right)\right)^l D_{\mathrm{KL}}(\delta_i \| \overline{\pi}) = - \left(1-\alpha\left(\overline{P}^2\right)\right)^l  \log (\overline{\pi}(i)).
    \end{align*}
    where the inequality follows from a result of \cite[Remark on Page 725]{diaconis_logsobolev}. Multiplying by $\overline{\pi}(i)$ and summing $i$ from $1$ to $k$ gives the desired result:
    \begin{align*}
        D_{\mathrm{KL}}^{\overline\pi} \left( \overline{P}^l \| \overline{\Pi}\right) &= \sum_{i=1}^k \overline{\pi}(i) D_{\mathrm{KL}}(\delta_i \overline{P}^l \| \overline{\pi}) \leq \left(1-\alpha\left(\overline{P}^2\right)\right)^l H(\overline{\pi}).
    \end{align*}
\end{proof}

We now consider the special case where $\mathcal{X} = S \sqcup S',\ \mathcal{O}_1 = S,\ \mathcal{O}_2 = S'$. Under this setting, as $\overline{P}$ is a two-state transition matrix, its log-Sobolev constant is known exactly. This allows us to give a decay rate of the KL divergence $D_{\mathrm{KL}}^{\overline\pi} \left( \overline{P}^l \| \overline{\Pi}\right)$ and consequently $D_{\mathrm{KL}}^\pi \left((GPG)^l \| \Pi \right)$ by Proposition \ref{prop:GPGoverlineP KL}.

\begin{corollary}
    Suppose $\mathcal{X} = S \sqcup S',\ \mathcal{O}_1 = S,\ \mathcal{O}_2 = S'$, $\overline{\pi} = (\pi(S), \pi(S'))$ with $\pi(S) \leq \pi(S')$. Then the log-Sobolev constant $\alpha\left(\overline{P}^2\right)$ in Proposition \ref{prop:logsobolevbound} can be written as
    $$\alpha\left(\overline{P}^2\right) = \frac{\overline{P}^2(1,2)(1 - 2\pi(S))}{\pi(S')\log(\pi(S')/\pi(S))},$$
    shown in Corollary A.3 of \cite{diaconis_logsobolev}, and hence
    \begin{align*}
       D_{\mathrm{KL}}^\pi \left((GPG)^l \| \Pi \right) \leq \left(1 -  \frac{\overline{P}^2(1,2)(1 - 2\pi(S))}{\pi(S')\log(\pi(S')/\pi(S))}\right)^l \log 2,
    \end{align*}
    where we use $H(\overline{\pi}) \leq \log 2$.
\end{corollary}

Writing 
\begin{align} \label{eq:g(S)}
    g(S,P) := \frac{1}{\pi(S)\pi(S')} \sum_{\substack{x \in S\\ y \in S'}} \pi(x)P(x,y),
\end{align}
we note that
\begin{align*}
     \frac{\overline{P}^2(1,2)}{\pi(S')} &= \frac{1}{\pi(S')}  \overline{P}(1,2) \left(\overline{P}(1,1) + \overline{P}(2,2)\right)\\
     &= g(S,P)\left(2 - \overline{P}(1,2)\left(1+\frac{\pi(S)}{\pi(S')}\right)\right)\\
     &= g(S,P)(2 - g(S,P)),
\end{align*}
where the second line follows from $\overline{\pi}$-reversibility of $\overline{P}$.

\subsection{Decay rate of the KL divergence of $(GP)^l$ from $\Pi$ and $(PG)^l$}
Here, we show how the KL divergence of $(GP)^l$ and $(PG)^l$ from $\Pi$ can be related closely with that of $(GPG)^l$.

\begin{proposition}\label{prop:GPGPGKL}
    Given $P \in \mathcal{L}(\pi)$, $\mathcal{X} = \bigsqcup_{i=1}^k \mathcal{O}_i$ and any Gibbs orbit kernel $G$, for $l \in \mathbb{N}$ with $l \geq 2$, 
    $$D_{\mathrm{KL}}^\pi \left((GPG)^l \| \Pi \right) \leq D_{\mathrm{KL}}^\pi \left((PG)^l \| \Pi \right) = D_{\mathrm{KL}}^\pi \left((GP)^l \| \Pi \right) \leq D_{\mathrm{KL}}^\pi \left((GPG)^{l-1} \| \Pi \right).$$
\end{proposition}

\begin{proof}
    Proposition 5.3 in \cite{GPG} shows that for any $\pi$-stationary $P$,
    $$D_{KL}^\pi(P \| \Pi) \geq D_{KL}^\pi(GPG \| \Pi).$$
    The left-most inequality is then a direct consequence of the above result since
    \begin{align*}
        D_{\mathrm{KL}}^\pi \left((PG)^l \| \Pi \right) \geq D_{\mathrm{KL}}^\pi \left(G(PG)^lG \| \Pi \right) = D_{\mathrm{KL}}^\pi \left((GPG)^l \| \Pi \right).
    \end{align*}

    The right-most inequality can be derived from 
    $$D_{\mathrm{KL}}^\pi \left((PG)^l \| \Pi \right) = D_{\mathrm{KL}}^\pi \left(P(GPG)^{l-1} \| \Pi \right) \leq D_{\mathrm{KL}}^\pi \left((GPG)^{l-1} \| \Pi \right),$$
    where the last inequality follows from the weak data-processing inequality of Markov chains in Proposition 3.9 of \cite{wang2023informationdivergencesmarkovchains}.
    
    Finally, the equality is given by the bisection property, for $P,Q \in \mathcal{S}(\pi)$,
    $$D_{\mathrm{KL}}^\pi(P \| Q) = D_{\mathrm{KL}}^\pi(P^* \| Q^*),$$
    as given in \cite{wolfer_2023} Section 4.2.2.
\end{proof}

By combining Proposition \ref{prop:logsobolevbound} and Proposition \ref{prop:GPGPGKL}, we thus have the following:
\begin{corollary}
    For $P \in \mathcal{L}(\pi)$, $\mathcal{X} = \bigsqcup_{i=1}^k \mathcal{O}_i$ and any Gibbs orbit kernel $G$,  
    $$D_{\mathrm{KL}}^\pi \left((PG)^l \| \Pi \right) \leq \left(1-\alpha\left(\overline{P}^2\right)\right)^{l-1} H(\overline{\pi}),$$
    for $l \in \mathbb{N}$ with $l \geq 2$, and $\alpha$ as given in \eqref{eq:alpha}.
\end{corollary}

As such, the rate of decay of the KL divergence of both $(GP)^l$ and $(PG)^l$ from $\Pi$ can be bounded above by an expression that depends on the log-Sobolev constant of $\overline{P}^2$.

\section{Optimisation in Frobenius norm} \label{sec:frob}
First, we show the relationship between $GP^2$ and $\overline{P^2}$ in terms of trace. 

\begin{proposition} \label{prop:trace GP^2 = P^2bar}
        For $P\in \mathcal{S}(\pi)$ and $\mathcal{X} = \bigsqcup_{i=1}^k \mathcal{O}_i$, we have that 
        $$\Tr(GP^2) = \Tr(\overline{P^2}).$$
\end{proposition}

\begin{proof}
    \begin{align*}
        \Tr(GP^2) &= \sum_{x\in \mathcal{X}}GP^2(x,x)\\
        &= \sum_{i=1}^k \frac{1}{\pi(\mathcal{O}_i)} \sum_{x, y \in \mathcal{O}_i} \pi(x)P^2(x,y)\\
        &= \sum_{i=1}^k \overline{P^2}(i,i) = \Tr(\overline{P^2}).
    \end{align*}
\end{proof}

\begin{corollary}\label{cor:GPPi}
    For $P\in \mathcal{L}(\pi)$,
    \begin{equation*}
        \|GP- \Pi\|^2_{F,\pi} = \Tr(GP^2) - 1 = \Tr(\overline{P^2}) - 1 =  \Tr(\overline{P^2}- \overline{\Pi}).
    \end{equation*}
\end{corollary}

A similar relationship can be shown for $GPG$ and $\overline{P}$.
\begin{proposition}
    Given $P\in \mathcal{L}(\pi)$ and $\mathcal{X} = \bigsqcup_{i=1}^k \mathcal{O}_i$, for $l \in \mathbb{N}$ we have
    \begin{equation*}
        \Tr((GPG)^l) = \Tr\left(\overline{P}^l\right).
    \end{equation*}
\end{proposition}

\begin{proof}
Proposition 6.2 of \cite{GPG} shows that the non-zero eigenvalues of $GPG$ coincide with those of $\overline{P}$, and that all remaining eigenvalues of $GPG$ are equal to $0$.
Since $\Tr(A^l)=\sum_i \lambda_i(A)^l$ for any real, square matrix $A$ and integer $l \geq 1$, the claim is as follows.
\end{proof}

Analogous to Proposition \ref{prop:GPGoverlineP KL} which relates the KL divergence of $(GPG)^l$ from $\Pi$ to the KL divergence of $\overline{P}^l$ from $\overline{\Pi}$, using the above proposition allows us to reduce the Frobenius norm of $(GPG)^l - \Pi$ to the projected version $\overline{P}^l - \overline{\Pi}$:

\begin{corollary} \label{cor:norm GPG}
    For $P \in \mathcal{L}(\pi)$ and $\mathcal{X} = \bigsqcup_{i=1}^k \mathcal{O}_i$, it follows that for $l \in \mathbb{N}$
    \begin{align*}
        \|(GPG)^l - \Pi\|_{F, \pi}^2 &= \Tr((GPG)^{2l})-1 \nonumber\\
        &= \Tr(\overline{P}^{2l}) - \Tr(\overline{\Pi}) = \left\|\overline{P}^l - \overline{\Pi}\right\|^2_{F,\overline{\pi}}.
    \end{align*}
\end{corollary}

Analogous to Proposition \ref{prop:GPGPGKL}, we can further relate the rate of convergence in Frobenius norm to stationarity between $(GP)^l$, $(PG)^l$ and $(GPG)^l.$

\begin{proposition}
    For $P \in \mathcal{L}(\pi)$, $\mathcal{X} = \bigsqcup_{i=1}^k \mathcal{O}_i$ and any Gibbs kernel $G$, the inequality
    $$\|(GPG)^l - \Pi\|_{F, \pi} \leq \|(PG)^l - \Pi\|_{F, \pi} = \|(GP)^l - \Pi\|_{F, \pi}$$
    holds for $l \in \mathbb{N}$. For $l \geq 2$, we also have
    \begin{align*}
         \|(PG)^l - \Pi\|_{F, \pi} = \|(GP)^l - \Pi\|_{F, \pi} \leq \|(GPG)^{l-1} - \Pi\|_{F, \pi}.
    \end{align*}
\end{proposition}

\begin{proof}
    Let $H$ be a Hilbert space with $\mathcal{T}:H \to H$ a bounded linear operator and $\mathcal{S}: H \to H$ a Hilbert-Schmidt operator. Then we have the following properties (see \cite{conwayfunctional}):
    \begin{align*}
        \|\mathcal{S}^*\|_{\mathrm{HS}} &= \|\mathcal{S}\|_{\mathrm{HS}},\\
        \|\mathcal{TS}\|_{\mathrm{HS}} &\leq \|\mathcal{T}\|_{\mathrm{op}}\|\mathcal{S}\|_{\mathrm{HS}},\\
        \|\mathcal{ST}\|_{\mathrm{HS}} &\leq \|\mathcal{S}\|_{\mathrm{HS}}\|\mathcal{T}\|_{\mathrm{op}},
    \end{align*}
    where $\| \cdot \|_{\mathrm{op}}$ is the usual operator norm and $\|\cdot \|_{\mathrm{HS}}$ is the Hilbert-Schmidt norm. 

    Since the Frobenius norm is a Hilbert-Schmidt norm on a finite dimensional space, all $P \in \mathcal{S}(\pi)$ are bounded operators. Furthermore, $\|P\|_{\mathrm{op}}^2 \leq 1$ for all $P \in \mathcal{L}(\pi)$.

    Hence the first inequality holds by noting that 
    $$\|(GPG)^l - \Pi\|_{F, \pi} = \|G((PG)^l - \Pi)\|_{F, \pi} \leq \|(PG)^l - \Pi\|_{F, \pi}.$$
    The second inequality is similar, with
    $$\|(PG)^l - \Pi\|_{F, \pi} = \|P((GPG)^{l-1} - \Pi)\|_{F, \pi} \leq \|(GPG)^{l-1} - \Pi\|_{F, \pi},$$
    and finally the equality holds since $((PG)^l - \Pi)^* = (GP)^l - \Pi$ given that $P \in \mathcal{L}(\pi).$
\end{proof}

\subsection{Case of $k=2$ orbits for $GP$} \label{sec:GP_k=2}
We consider in this subsection the case where $\mathcal{X} = S \sqcup S'$, and we seek to optimise with respect to $S$ the objective $f : 2^{\mathcal{X}} \to \mathbb{R}^+$ given by
\begin{align}\label{eq:f(S)}
    f(S) = f(S,P) := \|G_S P - \Pi\|_{F, \pi}^2,
\end{align}
where we use $G = G_S$ to highlight the dependence of the Gibbs kernel with respect to the choice of $S$. 

First, we exclude the two trivial cases $S = \mathcal{X}$ and $S = \emptyset$, since the choice of such partitions give $G_S = \Pi$ and hence $\|G_S P - \Pi\|_{F, \pi}^2 =0$. Therefore, we look towards solving the objective
\begin{equation}\label{eq:GP obj}
    \min_{S \neq \mathcal{X}, \emptyset} \|G_S P - \Pi\|_{F, \pi}^2.
\end{equation}

Recall $g : 2^{\mathcal{X}} \to \mathbb{R}^+$ as given in \eqref{eq:g(S)}, where we instead consider $P^2$:
\begin{equation*}
    g(S) = g(S,P^2) := \frac{1}{\pi(S)\pi(S')} \sum_{\substack{x \in S\\ y \in S'}} \pi(x)P^2(x,y).
\end{equation*}

For $P \in \mathcal{L}(\pi)$, in the special case when one takes $S^*$ to be the symmetrised Cheeger's cut of $P^2$, then we have
\begin{align*}
    g(S,P^2) \geq g(S^*, P^2) = \phi^*(P^2).
\end{align*}

Our first result of this section links the squared-Frobenius norm to stationarity of $G_S P$ to $g(S)$, and in the special case when one takes $S$ to be a symmetrised Cheeger's cut of $P^2$, the same norm can be linked to the symmetrised Cheeger's constant.
\begin{proposition} \label{prop:GP argmin}
    For any $\pi$-reversible $P$, $S \neq \mathcal{X}, \emptyset$ and $\mathcal{X} = S \sqcup S'$, we have
    \begin{align*}
        \|G_S P - \Pi\|_{F, \pi}^2 = 1 - g(S),
    \end{align*}
    where we recall $g$ is introduced in \eqref{eq:g(S)}. In particular, when one takes $S^*$ to be a symmetrised Cheeger's cut of $P^2$, then we have
    \begin{align*}
        \|G_{S^*} P - \Pi\|_{F, \pi}^2 = 1 - \phi^*(P^2).
    \end{align*}
    We also note that the non-trivial set $S \subset \mathcal{X}$ that achieves the minimum of \eqref{eq:GP obj} also maximises the function $g(S)$. That is, 
    $$\argmin_{S \neq \mathcal{X}, \emptyset} \|G_S P - \Pi\|_{F, \pi}^2 = \argmax_{S \neq \mathcal{X}, \emptyset} g(S).$$
    In fact, the other direction holds as well:
    $$\argmax_{S \neq \mathcal{X}, \emptyset} \|G_S P - \Pi\|_{F, \pi}^2 = \argmin_{S \neq \mathcal{X}, \emptyset} g(S).$$
\end{proposition}

\begin{proof} \label{prop:GP argmax}
    From Proposition \ref{prop:trace GP^2 = P^2bar}, for any non-empty $S \neq \mathcal{X}$,
    \begin{align*}
        \Tr(GP^2) &= \Tr(\overline{P^2})\\
        &= \frac{1}{\pi(S)} \sum_{x,y \in S} \pi(x) P^2(x,y) + \frac{1}{\pi(S')} \sum_{x,y \in S'} \pi(x) P^2(x,y)\\
        &= \frac{1}{\pi(S)} \bigg(\pi(S) -  \sum_{\substack{x \in S \\ y \in S'}} \pi(x) P^2(x,y)\bigg) + \frac{1}{\pi(S')} \bigg(\pi(S') -  \sum_{\substack{x \in S' \\ y \in S}} \pi(x) P^2(x,y)\bigg)\\
        &= 2 - \frac{1}{\pi(S) \pi(S')} \sum_{\substack{x \in S \\ y \in S'}} \pi(x) P^2(x,y).
    \end{align*}
    Hence, 
    $$\|G_S P - \Pi\|_{F, \pi}^2 = 1 - \frac{1}{\pi(S) \pi(S')} \sum_{\substack{x \in S \\ y \in S'}} \pi(x) P^2(x,y)$$
    and the claim holds. 
\end{proof}

In fact, we note that the minimisation need only be done over the space
$$\mathcal{A} = \{S \subset \mathcal{X}: 0 < \pi(S) \leq 1/2\}.$$

\begin{proposition}
    The two optimisation problems below are equal given any $P \in \mathcal{L}(\pi)$, $S \neq \mathcal{X}, \emptyset$ and $\mathcal{X} = S \sqcup S'$:
    $$\min_{S \neq \mathcal{X}, \emptyset} \|G_S P - \Pi\|_{F, \pi}^2 = \min_{S \in \mathcal{A}} \|G_S P - \Pi\|_{F, \pi}^2.$$
\end{proposition}

\begin{proof}
    Notice that $g(S)$ given in \eqref{eq:g(S)} is symmetric in the sense that $g(S) = g(S')$. Hence for any $S \subset \mathcal{X}$ with $1/2 < \pi(S) < 1$, it must hold that $f(S) = 1-g(S) = f(S').$ It thus suffices to minimise the objective over $\mathcal{A}$.
\end{proof}

Proposition \ref{prop:GP argmax} therefore reveals that the partitions induced by Cheeger's cuts are the worst possible choices of orbits for the choice of $S$ when measured in the Frobenius norm. From this perspective, the objective can be understood as ``anti-Cheeger'', it selects orbits that deliberately cut across metastable regions rather than respecting them.

However, direct optimisation of $g(S)$ is combinatorial and often intractable when $|\mathcal{X}|$ is large. This motivates us to consider the functional 

\begin{equation}\label{eq:h(S)}
    h(S) = h(S,P^2) := \frac{1}{\pi(S)} \sum_{\substack{x \in S\\ y \in S'}} \pi(x)P^2(x,y).
\end{equation}

We first show how we can always find a set $S \in \mathcal{A}$ which maximises $h(S)$, before showing how it can be used to approximate a solution in minimising $\|G_S P - \Pi\|_{F, \pi}^2$.

\begin{proposition} \label{prop:x* optimal}
    For ergodic $P \in \mathcal{S}(\pi)$ and $|\mathcal{X}| \geq 2$, let 
    $$x^* = x^*(P^2) \in \argmax_{0 < \pi(x) < 1} \left(1-P^2(x,x)\right).$$
    Then 
    $$U^* = \{x^*\} \in \argmax_{S \neq \mathcal{X}, \emptyset} h(S).$$
    Note that $U^* \in \mathcal{A}$.
\end{proposition}

\begin{proof}
    For some non-empty $S \neq \mathcal{X}$, write 
    $$h(S) = \frac{1}{\pi(S)} \sum_{\substack{x \in S\\ y \in S'}} \pi(x)P^2(x,y) = \sum_{x \in S} \frac{\pi(x)}{\pi(S)} P^2(x,S').$$
    Then 
    \begin{align*}
        h(S) &\leq \max_{x \in S} P^2(x, S')\\
        &= \max_{x \in S} 1-P^2(x,S) \\
        &\leq \max_{x \in S} 1 - P^2(x,x).
    \end{align*}
    Now let $x^* \in \argmax_{0 < \pi(x) < 1} \left(1-P^2(x,x)\right).$ It must then be the case that 
    $$h(S) \leq 1 - P^2(x^*, x^*) = h(U^*),$$
    achieving the upper bound as claimed.
    
    Note that $h(U^*) > 0$. If $h(U^*) = 0$, then we have $P^2(x,x) = 1$ for all $x \in \mathcal{X}$. That is, $P^2 = I$, which contradicts that $P$ is ergodic.
    Assume the contrary such that $U^* \notin \mathcal{A}$, that is, $\pi(U^*) > \frac{1}{2}$. Consider
    \begin{align*}
        h(U^{*\prime}) &= \frac{1}{\pi(U^{*'})} \sum_{\substack{x \in U^{*'}\\ y \in U^{*}}} \pi(x)P^2(x,y) \\
        &= \frac{1}{\pi(U^{*'})} \sum_{\substack{x \in U^{*}\\ y \in U^{*'}}} \pi(x)P^2(x,y) \\
        &> \frac{1}{\pi(U^{*})} \sum_{\substack{x \in U^{*}\\ y \in U^{*'}}} \pi(x)P^2(x,y) = h(U^*),
    \end{align*}
    which contradicts $U^* \in \argmax_{S \neq \mathcal{X}, \emptyset} h(S)$. Note that the second equality above uses $P^2 \in \mathcal{S}(\pi)$, and the inequality makes use of $h(U^*) > 0$. 

\end{proof}

\begin{lemma}\label{lem:hg2h}
    Let $P \in \mathcal{S}(\pi)$. For any $S \in \mathcal{A}$, that is, $0 < \pi(S) \leq 1/2,$
    \begin{equation*}
        h(S) \leq g(S) \leq 2h(S),
    \end{equation*}
    where we recall $g, h$ are introduced in \eqref{eq:g(S)} and \eqref{eq:h(S)} respectively.
\end{lemma}

\begin{proof}
    The inequality holds from the fact that 
    $$g(S) = \frac{h(S)}{\pi(S')}\quad \mathrm{and}\quad 1/2 \leq \pi(S') < 1.$$
\end{proof}

\begin{proposition} \label{prop:1/2-approx}
    Consider some ergodic $P \in \mathcal{S}(\pi)$ and let $|\mathcal{X}| \geq 2$. Recall that $f, h$ are introduced in \eqref{eq:f(S)} and \eqref{eq:h(S)} respectively. Any solution
    $$U^* \in \argmax_{S\in \mathcal{A}} h(S)$$
    is an additive $\frac{1}{2}$-approximate minimiser of $f(S)$. That is, 
    $$S^* \in \argmin_{S\in \mathcal{A}} f(S)$$ 
    satisfies 
    $$f(U^*) - f(S^*) \leq \frac{1-f(S^*)}{2} \leq \frac{1}{2}.$$
\end{proposition}

\begin{proof}
    By Lemma \ref{lem:hg2h}, we note that 
    \begin{equation} \label{eq:s* u*}
        f(S^*) = 1-g(S^*) \geq 1-2h(S^*), \quad f(U^*) \leq 1-h(U^*).
    \end{equation}
    By the definition of $U^*$ and \eqref{eq:s* u*} we thus have
    $$h(U^*) \geq h(S^*) \geq \frac{1-f(S^*)}{2}.$$
    This yields the left-most inequality in the statement:
    $$f(U^*) - f(S^*) \leq 1-h(U^*)-f(S^*) \leq \frac{1-f(S^*)}{2}.$$

    Noting that $f(S) = \|G_S P - \Pi\|_{F, \pi}^2 \geq 0$, the right-most inequality holds as well.  
\end{proof}

With Propositions \ref{prop:x* optimal} and \ref{prop:1/2-approx}, we show that while direct optimisation of the objective $\min_{S \neq \mathcal{X}, \emptyset} \|G_S P - \Pi\|_{F, \pi}^2$ is in general NP-hard, an approximate minimiser for the objective can be achieved by a singleton set consisting only of $x^*$ which maximises $1-P^2(x,x)$. This effectively reduces the search space from the problem $\argmax_{S \neq \mathcal{X}, \emptyset} g(S)$ which is exponential in $n = |\mathcal{X}|$ to $\argmax_{0 < \pi(x) < 1} \left(1-P^2(x,x)\right)$ which would be linear in $n$.

Lastly, in the case where we restrict $S$ to only singletons, we have the following result:

\begin{proposition}\label{prop:|S|=1GP}
    For $P \in \mathcal{L}(\pi)$, under the additional constraint $|S| = 1$, 
    $$\argmin_{\substack{S \neq \mathcal{X}\\ |S| = 1}} f(S) = \argmax_{\substack{S \neq \mathcal{X}\\ |S| = 1}} g(S) = \argmax_{x \in \mathcal{X}} \frac{1-P^2(x,x)}{1-\pi(x)},$$
    where we recall that $f, g$ are introduced in \eqref{eq:f(S)} and \eqref{eq:g(S)} respectively.
\end{proposition}

\begin{proof}
    This follows from the definition of $g(S)$ by setting $S = \{x\}$ for some $x \in \mathcal{X}$.
\end{proof}

\subsection{Case of $k=2$ for $GPG$} \label{sec:GPG_k=2}
In this subsection, we again let $\mathcal{X} = S \sqcup S'$, and we wish to minimise with respect to $S$ the objective $\|G_S PG_S - \Pi\|_{F, \pi}^2$. 

In the first result of this subsection, we give an upper bound on the discrepancy $(G_SPG_S)^l - \Pi$ under the Frobenius norm:
\begin{proposition}\label{prop:GSPGSlPi}
    For $P \in \mathcal{L}(\pi)$ and some subset $S \subset \mathcal{X}$, $S \neq \mathcal{X}, \emptyset$, $\mathcal{X} = S \sqcup S'$, it holds that for $l \in \mathbb{N}$,
    \begin{align*}
        \|(G_SPG_S)^l - \Pi\|_{F, \pi}^2 &= \lambda_2(G_SPG_S)^{2l} \\
        &= \lambda_2(\overline{P})^{2l} 
        = \left(1 - \frac{1}{\pi(S)\pi(S')}\sum_{\substack{x \in S\\ y \in S'}}  \pi(x)P(x,y) \right)^{2l}.
    \end{align*}
    In particular, when $P$ is additionally assumed to be positive-semidefinite, we have
    \begin{align*}
        \|(G_SPG_S)^l - \Pi\|_{F, \pi} &= \left(1 - \frac{1}{\pi(S)\pi(S')}\sum_{\substack{x \in S\\ y \in S'}}  \pi(x)P(x,y) \right)^{l} = \left(1-g(S,P)\right)^l,
    \end{align*}
    where we recall $g(S,P)$ is defined in \eqref{eq:g(S)}.
\end{proposition}

\begin{proof}
    By Proposition 6.2 of \cite{GPG}, the eigenvalues of $G_S P G_S$ are exactly 
    \begin{align*}
        &\lambda_1(G_SPG_S) = 1,\ \lambda_2(G_SPG_S) = \lambda_2(\overline{P}),\\ &\lambda_3(G_SPG_S) = \ldots = \lambda_n(G_SPG_S) = 0.
    \end{align*}
    It follows that the only non-zero eigenvalue of $(G_SPG_S)^l - \Pi$ is $\lambda_2(\overline{P})^l$, and by the definition of the Frobenius norm given that $G_SPG_S \in \mathcal{L}(\pi)$, 
    $$\|(G_SPG_S)^l - \Pi\|_{F, \pi}^2 = \Tr\left(((G_SPG_S)^l - \Pi)^2\right) = \lambda_2(\overline{P})^{2l}.$$

    One can then verify that 
    \begin{align*}
        \lambda_2(\overline{P}) &= \frac{1}{\pi(S)} \sum_{x, y \in S} \pi(x) P(x,y) + \frac{1}{\pi(S')} \sum_{x, y \in \mathcal{S'}} \pi(x) P(x,y) - 1\\
        &= 1 - \frac{1}{\pi(S)}\sum_{\substack{x \in S\\ y \in S'}}  \pi(x)P(x,y) - \frac{1}{\pi(S')}\sum_{\substack{x \in S'\\ y \in S}} \pi(x)P(x,y)\\
        &= 1 - \frac{1}{\pi(S)}\sum_{\substack{x \in S\\ y \in S'}}  \pi(x)P(x,y) - \frac{1}{\pi(S')}\sum_{\substack{x \in S'\\ y \in S}} \pi(y)P(y,x)\\
        &= 1 - \frac{1}{\pi(S)\pi(S')}\sum_{\substack{x \in S\\ y \in S'}}  \pi(x)P(x,y),
    \end{align*}
    where we use the $\pi$-reversibility of $P$ in the third equality.
    
    If $P$ is positive-semidefinite, then for any $f \in \ell^2(\pi)$, 
    $$\langle f, G_S P G_S f\rangle_\pi = \langle G_S f, PG_S f \rangle_\pi \geq 0,$$
    which implies that $G_SPG_S$ is also positive-semidefinite. It follows that $0 \leq \lambda_2(G_S P G_S) \leq 1$ and the last part of the claim holds. 
\end{proof}

We thus see that, for $l \in \mathbb{N}$, we have
\begin{align*}
    \argmin_{S \neq \mathcal{X}, \emptyset} \|(G_SPG_S)^l - \Pi\|_{F, \pi}^2 = \argmin_{S \neq \mathcal{X}, \emptyset} \|G_SPG_S - \Pi\|_{F, \pi}^2.
\end{align*}
Therefore, in the remainder of this paper, the above equality justifies the decision for us to focus on minimising the one-step deviation from stationarity, $G_S P G_S - \Pi$, measured in the squared-Frobenius norm.

In the special case of $l = 1$, we obtain the following corollary:
\begin{corollary} \label{cor:GPG frob}
    For any $\pi$-reversible $P$ and fixed $S \subset \mathcal{X}$, $S \neq \mathcal{X}, \emptyset$, we have
    $$\|G_SPG_S - \Pi\|_{F, \pi}^2 = \left(1 - \frac{1}{\pi(S)\pi(S')}\sum_{\substack{x \in S\\ y \in S'}}  \pi(x)P(x,y)\right)^2.$$
    In particular, when one takes $S^*$ to be a symmetrised Cheeger's cut of $P$, then we have
    \begin{align*}
        \|G_{S^*} P G_{S^*} - \Pi\|_{F, \pi}^2 = \left(1 - \phi^*(P)\right)^2.
    \end{align*}
\end{corollary}

The result above shows that the Frobenius norm of $G_S P G_S - \Pi$ depends explicitly on the one-step transition kernel $P$, in contrast to the objective for $G_S P - \Pi$ in Proposition \ref{prop:GP argmin}, which depends on the two-step transition $P^2$. This difference can be traced to the identities
$$\|G_SP\|_{F, \pi}^2 = \Tr(G_S P^2)=\Tr(\overline{P^2}) \quad \text{and} \quad \|G_SPG_S\|_{F, \pi}^2 = \Tr((G_S P G_S)^2)=\Tr\left(\overline{P}^2\right),$$
so that the former aggregates two-step transition behaviour before projection, while the latter projects the one-step dynamics directly. We note that in general, $\overline{P^2}\neq \overline{P}^2$, and hence these two quantities encode different information about the underlying chain.

Analogous to Proposition \ref{prop:|S|=1GP} for $G_S P$, we consider the case of restricting $S$ to be a singleton in the context of $G_S P G_S$:
\begin{corollary}
    For $P \in \mathcal{L}(\pi)$, under the additional constraint of $|S| = 1$, 
    $$\argmin_{\substack{S \neq \mathcal{X}, \emptyset\\ |S|=1}} \|G_SPG_S - \Pi\|_{F, \pi}^2 = \argmin_{x \in \mathcal{X}} \bigg(\frac{P(x,x) - \pi(x)}{1-\pi(x)}\bigg)^2.$$
\end{corollary}

Collecting Proposition \ref{prop:GP argmin} and Proposition \ref{prop:GSPGSlPi}, we see that for positive-semidefinite $P \in \mathcal{L}(\pi)$, we have
\begin{align*}
    \|G_SP - \Pi\|_{F, \pi}^2 &= 1 - g(S,P^2), \\
    \|G_SPG_S - \Pi\|_{F, \pi} &= 1 - g(S,P). 
\end{align*}
Interestingly, the former is governed by $P^2$ while the latter depends on $P$. Therefore, replacing $P^2$ by $P$ in the analysis that leads to Proposition \ref{prop:1/2-approx} yields a $\frac{1}{2}$-approximate minimiser of $\|G_SPG_S - \Pi\|_{F, \pi}$. The proof is omitted as it is the same as that of Proposition \ref{prop:1/2-approx} with $P^2$ therein replaced by $P$.
\begin{proposition} \label{prop:1/2-approxGPG}
    Assume that $P$ is $\pi$-reversible, ergodic and positive-semidefinite, $|\mathcal{X}| \geq 2$ and recall that $h$ is introduced in \eqref{eq:h(S)}. Any solution
    $$U^* = \{x^*(P)\} \in \argmax_{S\in \mathcal{A}} h(S,P)$$
    with $x^*(P)$ as in Proposition \ref{prop:x* optimal} is an additive $\frac{1}{2}$-approximate minimiser of $\|G_SPG_S - \Pi\|_{F, \pi}$. That is, 
    $$S^* \in \argmin_{S\in \mathcal{A}} \|G_S P G_S - \Pi\|_{F, \pi}$$ 
    satisfies 
    $$\|G_{U^*}PG_{U^*} - \Pi\|_{F, \pi} - \|G_{S^*}PG_{S^*} - \Pi\|_{F, \pi} \leq \frac{1}{2}.$$
\end{proposition}

\subsection{A recursive construction for minimising the Frobenius norm}
Here, we describe a recursive construction that sequentially isolates singletons and applies
group-averaging on the remaining subspace. The construction produces a sequence of Frobenius distances computed on nested reduced state spaces, together with a recursion that bounds the distance at level $i$ in terms of the distance at level $i+1$.

Let the state space be $\mathcal{X} = \llbracket n\rrbracket$, and let $P$ be a $\pi$-reversible, ergodic, and positive-semidefinite Markov kernel on $\mathcal{X}$. At the first step, we partition $\mathcal{X}$ into
$$\mathcal{X}=S_1\sqcup S_1', \qquad S_1=\{1\}.$$

By convention, we write $\pi_1 = \pi.$ With Proposition \ref{prop:GSPGSlPi}, the Frobenius distance associated with this partition satisfies
$$\|G_{S_1}PG_{S_1}-\Pi_1\|_{F,\pi_1} = \frac{P(1,1)-\pi_1(1)}{1-\pi_1(1)}.$$

We now restrict attention to the complement $S_1'$, endowed with the conditional distribution
$$\pi_2(x) := \frac{\pi_1(x)}{1-\pi_1(1)}, \qquad x\in S_1'.$$

Repeating the same construction on $S_1'$, we split
$$S_1' = S_2 \sqcup S_2', \qquad S_2=\{2\},$$
and for any $\pi_2$-reversible, ergodic, positive-semidefinite sampler $P_2$, we have
$$\|G_{S_2}P_2G_{S_2}-\Pi_2\|_{F,\pi_2} = \frac{P_2(2,2)-\pi_2(2)}{1-\pi_2(2)},$$
where for any $i \in \llbracket n-1 \rrbracket$, $\Pi_i$ is the transition matrix whose rows are all $\pi_i.$

The recursive nature of the construction follows from the fact that $G_{S_1}$ can be written as the limit of the next constructed $G_{S_2}P_2G_{S_2}$. Concretely,

$$G_{S_1} =
\begin{pmatrix}
1 & 0\\
0 & \Pi_2
\end{pmatrix}
=
\lim_{l\to\infty}
\begin{pmatrix}
1 & 0\\
0 & (G_{S_2}P_2G_{S_2})^l
\end{pmatrix}.
$$

Defining
\begin{equation}\label{eq:A_l}
    A_l := \begin{pmatrix}
    1 & 0\\
    0 & (G_{S_2}P_2G_{S_2})^l
    \end{pmatrix}, 
\end{equation}
we may write
$$\|A_l P A_l - \Pi_1\|_{F,\pi_1} = \|A_l P (A_l-G_{S_1}) + (A_l-G_{S_1})PG_{S_1} + G_{S_1}PG_{S_1}-\Pi_1\|_{F,\pi_1},
$$
which tends towards the limit $\|G_{S_1}PG_{S_1}-\Pi_1\|_{F,\pi_1}$ as $l$ approaches $\infty$.

\begin{lemma} \label{lemma:A_l-G_S}
    Let $A_l$ be defined as per \eqref{eq:A_l} with $P_2$ being $\pi_2$-reversible. Then 
    $$\| A_l - G_{S_1}\|_{F,\pi_1}^2 = \|(G_{S_2}P_2G_{S_2})^l - \Pi_2\|_{F, \pi_2}^2.$$
\end{lemma}

\begin{proof}
    Since 
    $$A_l - G_{S_1} = \begin{pmatrix}
    0 & 0\\
    0 & (G_{S_2}P_2G_{S_2})^l - \Pi_2
    \end{pmatrix}$$
    it follows that 
    \begin{align}
        \| A_l - G_{S_1}\|_{F,\pi_1}^2 &= \sum_{x,y \in S_1'} \frac{\pi_1(x)}{\pi_1(y)}\left((G_{S_2}P_2G_{S_2})^l(x,y) - \Pi_{S_1'}(x,y)\right)^2\\
        &= \sum_{x,y \in S_1'} \frac{\pi_2(x)}{\pi_2(y)}\left((G_{S_2}P_2G_{S_2})^l(x,y) - \Pi_{S_1'}(x,y)\right)^2 \\
        &= \|(G_{S_2}P_2G_{S_2})^l - \Pi_{2}\|_{F, \pi_2}^2.
    \end{align}
\end{proof}

With this, we obtain an upper bound for $\|A_l P_1 A_l -\Pi_1\|_{F, \pi_1}.$ 

\begin{proposition}\label{prop:normA_lPA_l}
    For $P_1$ that is $\pi_1$-reversible and $P_2$ which is $\pi_2$-reversible and positive-semidefinite, 
    $$\|A_l P_1 A_l -\Pi_1\|_{F, \pi_1} \leq 2\|(G_{S_2}P_2G_{S_2})^l - \Pi_{2}\|_{F, \pi_2} + \|G_{S_1}P_1G_{S_1} - \Pi_1\|_{F, \pi_1}.$$
\end{proposition}

\begin{proof}
    By the subadditivity of norms,
    $$\|A_l P_1 A_l - \Pi_1\|_{F,\pi_1} \leq \|A_l P_1 (A_l-G_{S_1})\|_{F,\pi_1} + \|(A_l-G_{S_1})P_1G_{S_1}\|_{F,\pi_1} + \|G_{S_1}P_1G_{S_1}-\Pi_1\|_{F,\pi_1}.$$

    Let $\sigma_1(M) \geq \sigma_2(M) \geq \cdots \geq \sigma_n(M)$ denote the singular values of a matrix $M \in \mathbb{R}^{n \times n}$. It then follows 
    \begin{align*}
        \|A_l P_1 (A_l-G_{S_1})\|_{F,\pi_1}^2 &= \Tr((A_l-G_{S_1}) P_1 A_l^2 P_1 (A_l-G_{S_1}))\\
        &= \Tr(P_1 A_l^2 P_1 (A_l-G_{S_1})^2)\\
        &\leq \sum_{i=1}^n \sigma_i\left(P_1 A_l^2 P_1\right) \ \sigma_i \left((A_l-G_{S_1})^2\right)\\
        & \leq \sum_{i=1}^n\sigma_i \left((A_l-G_{S_1})^2\right)\\
        &= \Tr\left((A_l-G_{S_1})^2\right) \\
        &= \| A_l - G_{S_1}\|_{F,\pi_1}^2,
    \end{align*}
    where the first inequality follows from Von Neumann's trace inequality (a proof is given in \cite{mirsky}). 

    Equivalently, $\|(A_l-G_{S_1})P_1G_{S_1}\|_{F,\pi_1}^2 \leq \| A_l - G_{S_1}\|_{F,\pi_1}^2$.  With the results of the preceding Lemma \ref{lemma:A_l-G_S} and the fact that $P_2$ is positive-semidefinite, the conclusion holds.
\end{proof}

We formally describe the recursive construction below.\\

\begin{proposition}
Assume $P$ is $\pi$-reversible, ergodic, and positive-semidefinite on
$\mathcal X=\llbracket n\rrbracket$.

Set $\mathcal X_1:=\mathcal X$, $\pi_1:=\pi$, and $P_1:=P$.
For $i \in \llbracket n-1 \rrbracket$, given the current state space $\mathcal X_i$ with stationary distribution $\pi_i$ and Markov kernel $P_i$, perform:
\begin{enumerate}
    \item Choose a singleton
    $$S_i=\{x_i\}\subset \mathcal X_i \quad\text{with}\quad x_i\in\argmin_{x\in\mathcal X_i} \frac{P_i(x,x)-\pi_i(x)}{1-\pi_i(x)}.
    $$
    \item Let $\mathcal X_{i+1}:=\mathcal X_i\setminus S_i$ and
    set $\pi_{i+1}$ to be the conditional distribution on $\mathcal X_{i+1}$:
    $$
    \pi_{i+1}(y):=\frac{\pi_i(y)}{\pi_i(\mathcal X_{i+1})},\qquad y\in\mathcal X_{i+1}.
    $$
    \item Set $P_{i+1}$ to be any $\pi_{i+1}$-reversible, ergodic, and positive-semidefinite sampler of $\mathcal{X}_{i+1}$ (e.g. a natural restriction or a freshly constructed sampler).
\end{enumerate}
Then at each iteration $i$, the chosen singleton $S_i=\{x_i\}$ attains the smallest Frobenius distance from stationarity among all singleton splits of $\mathcal X_i$.
\end{proposition}

\begin{proof}
Fix $i$. For any singleton $S=\{x\}\subset\mathcal X_i$, apply
Proposition~\ref{prop:GSPGSlPi} on the state space $\mathcal X_i$ with
stationary distribution $\pi_i$ and kernel $P_i$ to obtain
$$
\|G_{\{x\}}P_iG_{\{x\}}-\Pi_{\mathcal X_i}\|_{F,\pi_i}
=\frac{P_i(x,x)-\pi_i(x)}{1-\pi_i(x)}.
$$
The algorithm chooses $x_i$ to minimise the left-hand side over $x\in\mathcal X_i$,
hence $S_i=\{x_i\}$ achieves the smallest Frobenius distance among all singleton splits
at iteration $i$.
\end{proof}

We emphasise that each of the Frobenius distance is taken with respect to a different state space, and accordingly, a different probability distribution. 

In fact, the results of Proposition \ref{prop:normA_lPA_l} can be generalised for each iteration of the algorithm described above.

\begin{corollary}
    For any $i \in \llbracket n-1\rrbracket$, define
    $$ A_l^{(i)} := \begin{pmatrix}
    1 & 0\\
    0 & (G_{S_{i+1}}P_{i+1}G_{S_{i+1}})^l
    \end{pmatrix}.$$
    Then $$G_{S_i} = \lim_{l \to \infty} A^{(i)}_l$$ and 
    $$\big\|A_l^{(i)} P_i A_l^{(i)} -\Pi_i\big\|_{F, \pi_i} \leq 2\|(G_{S_{i+1}}P_{i+1}G_{S_{i+1}})^l - \Pi_{i+1}\|_{F, \pi_{i+1}} + \|G_{S_i}P_iG_{S_i} - \Pi_i\|_{F, \pi_i}.$$
\end{corollary}

The proof follows exactly as that of Lemma \ref{lemma:A_l-G_S} and Proposition \ref{prop:normA_lPA_l} with an appropriate change of index.

\section{Reducing the squared-Frobenius norm to $\mathcal{O}(k)$ via $GP$ and $GPG$ with $k$ orbits} \label{sec:frobeg}

In this section, suppose that the state space $\mathcal{X} = \llbracket n \rrbracket$ is partitioned into $k \in \mathbb{N}$ orbits such that
\begin{align*}
    \mathcal{X} = \bigsqcup_{i=1}^k \mathcal{O}_i,
\end{align*}
where $\mathcal{O}_i \neq \emptyset$ for $i \in \llbracket k \rrbracket$, and we consider the induced Gibbs kernel $G = G(\pi,\mathcal{O}_1,\ldots,\mathcal{O}_k)$.

The main result of this section gives that, for $P \in \mathcal{L}(\pi)$, $\|GP - \Pi \|_{F,\pi}^2$ and $\|GPG - \Pi \|_{F,\pi}^2$ are of the order $\mathcal{O}(k)$, while if in addition $P$ is lazy, then $\|P - \Pi\|_{F,\pi}^2 = \Omega(n)$. Thus, the interpretation is that orbit-averaging such as $GP$ or $GPG$ reduces the Frobenius norm to stationarity to an order at most $k$, even when the original $P$ can be ``far" from stationarity as measured in Frobenius norm (e.g. of the order at least $n$ for lazy $P$).

\begin{proposition}\label{prop:Froborder}
    Consider $\mathcal{X} = \llbracket n \rrbracket$, the partition into $k$ orbits given by $\mathcal{X} = \bigsqcup_{i=1}^k \mathcal{O}_i$ where $\mathcal{O}_i \neq \emptyset$ for $i \in \llbracket k \rrbracket$,   and the induced Gibbs kernel $G$ with respect to the distribution $\pi$ and these orbits. 
    For $P \in \mathcal{L}(\pi)$, then we have
    \begin{align*}
        \|GP - \Pi \|_{F,\pi}^2 &\leq k - 1 = \mathcal{O}(k), \\
        \|GPG - \Pi \|_{F,\pi}^2 &\leq k - 1 = \mathcal{O}(k).
    \end{align*}
    In particular, when $k = 2$, we have
    \begin{align*}
        \|GP - \Pi \|_{F,\pi}^2 &\leq 1 = \mathcal{O}(1), \\
        \|GPG - \Pi \|_{F,\pi}^2 &\leq 1 = \mathcal{O}(1).
    \end{align*}
    In contrast, if $P \in \mathcal{S}(\pi)$ and is lazy, we have
    \begin{align*}
        \|P - \Pi\|_{F,\pi}^2 \geq \dfrac{n}{4} - 1 = \Omega(n).
    \end{align*}
\end{proposition}

\begin{proof}
    First, we note that for a transition matrix $Q \in \mathcal{L}(\mathcal{X})$, its trace is always bounded above by $n$, the size of the state space. That is, we have
    \begin{align}\label{eq:trQ}
        \Tr(Q) = \sum_{i=1}^n Q(i,i) \leq n = |\mathcal{X}|.
    \end{align}
    Using Corollary \ref{cor:GPPi}, \eqref{eq:trQ} and noting that $\overline{P^2} \in \mathcal{L}(\llbracket k \rrbracket)$ is of size $k \times k$, we thus have
    \begin{align*}
        \|GP- \Pi\|^2_{F,\pi} = \Tr(\overline{P^2}) - 1 \leq k - 1.
    \end{align*}
    Similarly, using Corollary \ref{cor:norm GPG}, \eqref{eq:trQ} and noting that $(\overline{P})^2 \in \mathcal{L}(\llbracket k \rrbracket)$ is of size $k \times k$, we have
    \begin{align*}
        \|GPG- \Pi\|^2_{F,\pi} = \Tr((\overline{P})^2) - 1 \leq k - 1.
    \end{align*}
    Finally, we calculate the squared-Frobenius norm to stationarity of $P$:
    \begin{align*}
        \|P - \Pi\|_{F,\pi}^2 &= \Tr(P^*P) - 1 \\
                              &\geq \sum_{i=1}^n (P(i,i))^2 - 1 \\
                              &\geq \dfrac{n}{4} - 1.
    \end{align*}
\end{proof}

\subsection{Example: improving the lazy simple random walk on the $d$-dimensional hypercube via Cheeger's cut and $G_S P G_S$}

In this subsection, we give an example in which $G_S P G_S$ still significantly improves mixing over the original $P$, when the set $S$ is taken to be a Cheeger's cut of $P$. Intuitively, such choice of $S$ may not perform well in $G_S P G_S$ since $P$ does not hop between $S$ and $S'$ effectively. Thus, the moral of this example is that even in such ``worst-case" choice of $S$, the use of $G_S P G_S$ still improves upon $P$.

Specifically, we consider $\mathcal{X} = \{-1,+1\}^d$, the $d$-dimensional hypercube for positive integer $d \geq 3$, and we take $P$ to be the lazy simple random walk on $\mathcal{X}$, that is, with probability $1/2$ the walk stays at the same state while with probability $1/2$, one out of the $d$ coordinates are chosen uniformly-at-random and is flipped to the opposite spin. Clearly, $\pi$ in this setting is the discrete uniform distribution on the hypercube.

In such setup, the eigenvalues of $P$ are shown to be (see e.g. \cite[Lemma $5.2.17$]{roch_mdp_2024}), for $i \in \llbracket 0,d \rrbracket$,
\begin{align*}
    \dfrac{d-i}{d}, \quad \mathrm{with \, multiplicity } \binom{d}{i}.
\end{align*}
Thus, by the binomial theorem of matrices we see that, for $l \in \mathbb{N}$,
\begin{align*}
    \|P^l - \Pi \|^2_{F,\pi} &= \|(P - \Pi)^l\|^2_{F,\pi} \\
                             &= \Tr((P - \Pi)^{2l}) \\
                             &\geq \dfrac{d(d-1)}{2} \left( 1 - \dfrac{2}{d} \right)^{2l}.
\end{align*}
Now, using the inequality $\ln(1-x) \geq - \dfrac{x}{1-x}$ for $x \in (0,1)$, we have
\begin{align*}
    \left( 1 - \dfrac{2}{d} \right)^{2l} \geq \exp\left(-\dfrac{4l}{d-2}\right),
\end{align*}
and hence
\begin{align*}
    \|P^l - \Pi \|^2_{F,\pi} \geq \dfrac{d(d-1)}{2} \exp\left(-\dfrac{4l}{d-2}\right).
\end{align*}
In particular, this yields
\begin{align}\label{eq:hypercubeP}
    l = d&:\quad \|P^d - \Pi \|^2_{F,\pi} = \Omega(d^2), \\
    l = \dfrac{1}{2}d \ln d&:\quad \|P^{(1/2)d\ln d} - \Pi \|^2_{F,\pi} = \Omega(1). \label{eq:hypercubeP2} 
\end{align}
On the other hand, we now choose $S = \{x \in \{-1,+1\}^d;~x(1) = -1\}$, the set of configurations such that the first coordinate has a spin of $-1$. It can be shown that $S$ is a Cheeger's cut of $P$, see e.g. \cite[Section $5.3.2$]{roch_mdp_2024}. Clearly, $\pi(S) = \pi(S') = 1/2$, and by Proposition \ref{prop:GSPGSlPi} yields
\begin{align*}
    \|(G_S P G_S)^l - \Pi \|^2_{F,\pi} &= (1-g(S,P))^{2l} \\
                                       &= (1-2\Phi^*(P))^{2l} \\
                                       &\leq (1 - \gamma(P))^{2l} = \left(1 - \dfrac{1}{d}\right)^{2l},
\end{align*}
where the inequality follows from the Cheeger's inequality \eqref{eq:Cheegerineq}. In particular, we have
\begin{align}\label{eq:hypercubeGPG}
    l &= d:\quad \|(G_S P G_S)^d - \Pi \|^2_{F,\pi} = \mathcal{O}(1), \\
    l &= \dfrac{1}{2}d \ln d:\quad \|(G_S P G_S)^{(1/2)d\ln d} - \Pi \|^2_{F,\pi} = \mathcal{O}\left(\frac{1}{d}\right). \label{eq:hypercubeGPG2}
\end{align}

We are interested in the time $l = \frac{1}{2}d \ln d$ since the Markov chain $P$ exhibits total-variation cutoff at a cutoff time $\frac{1}{2}d \ln d$, see e.g. \cite[Remark $5.2.19$]{roch_mdp_2024}. Contrasting \eqref{eq:hypercubeP}, \eqref{eq:hypercubeP2} with \eqref{eq:hypercubeGPG}, \eqref{eq:hypercubeGPG2}, we thus see that such $G_S P G_S$ converges faster than that of $P$ as measured by the squared-Frobenius norm to stationarity up to time $l = \frac{1}{2} d \ln d$.

\section{Submodular optimisation problems in group-averaged Markov chains} \label{sec:opt}

Recall that we write $G_S$ to be the Gibbs kernel associated with the orbits $\mathcal{X} = S \sqcup S'$, where $S \subseteq \mathcal{X}$ satisfies $0 < \pi(S) < 1.$

In this section, we focus on recognizing submodular structures in natural combinatorial optimisation problems under the KL divergence or the Frobenius norm in the context of group-averaged Markov chains. Taking advantages on these structures, we design optimisation algorithms for the tuning of $G$.

\subsection{Submodular optimisation and minimisation of KL divergence}

In this subsection, we investigate the link between submodular optimisation and entropy with respect to the KL divergence. In particular, we shall prove that the KL divergence from $\Pi$ to $PG_S$ or $G_S P$ can be decomposed as a difference of two submodular functions. This decomposition is constructive, which then allows us to employ existing algorithms for minimization of difference of submodular functions to find optimal or near-optimal $S$, thus offering a method for tuning of $G_S$.

\begin{proposition}[$D_{\mathrm{KL}}^\pi(PG_S \| \Pi)$ can be decomposed as a difference of two submodular functions]\label{prop:DPGdiffsub}
    Let $P$ be a $\pi$-reversible Markov kernel on a finite state space $\mathcal X$. 
    For any partition $S,S'$ such that $\mathcal{X} = S \sqcup S'$ where $S \neq \emptyset,\mathcal{X}$, the KL divergence
    $$D_{\mathrm{KL}}^\pi(PG_S \| \Pi)$$
    can be written as a difference of two submodular set functions. Equivalently, it is also a difference of two supermodular functions. 

    That is, for $\phi : \mathbb R^+ \to \mathbb R,\ \phi(t) = t \log t$,
    $$D_{\mathrm{KL}}^\pi(PG_S \,\|\, \Pi) = T(S) - U(S)$$ 
    where $T, U$ are two supermodular functions given by
    \begin{align*}
        T(S) &= T(S,P,\pi) := \sum_{x \in \mathcal X} \pi(x)\big[\phi(P(x,S)) + \phi(P(x,S'))\big],\\
        U(S) &= U(S,\pi) := \phi(\pi(S)) + \phi(\pi(S')) = -H(\overline{\pi}).
    \end{align*}
\end{proposition}

\begin{proof}
    For $S \subset \mathcal X$, define
    $$P(x,S) := \sum_{y \in S} P(x,y), \qquad \pi(S) := \sum_{y \in S} \pi(y).$$
    Then
    $$(PG_S)(x,y) = 
        \begin{cases}
            \dfrac{\pi(y)}{\pi(S)} P(x,S), & y \in S, \\[6pt]
            \dfrac{\pi(y)}{\pi(S')} P(x,S'), & y \in S' .
        \end{cases}$$

    The KL divergence can then be expressed as
    \begin{align*}
        D_{\mathrm{KL}}^\pi(PG_S \,\|\, \Pi) &= \sum_{x \in \mathcal X} \pi(x) \Bigg[P(x,S) \log \frac{P(x,S)}{\pi(S)} + P(x,S') \log \frac{P(x,S')}{\pi(S')}\Bigg] \\
        &= T(S) - U(S).
    \end{align*}

    For each fixed $x$, the set functions $S \mapsto P(x,S)$ and $S \mapsto P(x,S')$ are non-negative and modular. By Proposition 6.1 of \cite{bach2013learningsubmodularfunctionsconvex}, since $\phi$ is convex, the compositions
    $$S \mapsto \phi(P(x,S)),
        \qquad
        S \mapsto \phi(P(x,S'))$$
    are supermodular. Since supermodularity is preserved under non-negative linear combinations, the function
    $$T(S) = \sum_{x \in \mathcal X} \pi(x)\big[\phi(P(x,S)) + \phi(P(x,S'))\big]$$
    is therefore supermodular.

    The remaining terms are
    $$U(S) = \phi(\pi(S)) + \phi(\pi(S')),$$
    which is also supermodular since $\pi(S)$ and $\pi(S')$ are modular.

    Consequently, $D_{\mathrm{KL}}^\pi(PG_S \,\|\, \Pi) = T(S) - U(S)$ can be recognized as a difference of two supermodular functions, and hence equivalently a difference of two submodular functions.
\end{proof}

\begin{corollary}
    Under the same settings as Proposition \ref{prop:DPGdiffsub}, it follows from the bisection property $D_{\mathrm{KL}}^\pi(P \| Q) = D_{\mathrm{KL}}^\pi(P^* \| Q^*)$ (e.g. shown in \cite{wolfer_2023} Section 4.2.2) that for $\pi$-reversible $P$, one can equivalently write
    $$D_{\mathrm{KL}}^\pi(G_SP \,\|\, \Pi) =  D_{\mathrm{KL}}^\pi(PG_S \,\|\, \Pi) = T(S) - U(S).$$
\end{corollary} 

In fact, if we consider $P = I$, the decomposition obtained in Proposition \ref{prop:DPGdiffsub} becomes
$$D_{\mathrm{KL}}^\pi(G_S \,\|\, \Pi) = - U(S) = H(\overline{\pi}).$$
In this special case, it is shown in Proposition 7.1 of \cite{GPG} that the optimal cut is given by $S = \{x_{\min}\}$, where $x_{\min}$ is any element with the smallest $\pi$ mass. That is, it is any $x \in \mathcal{X}$ satisfying $\pi(x) \leq \pi(y)$ for all $y \in \mathcal{X}$.

We now show that optimisation of the KL divergence can be seen equivalently as a maximum entropy problem: 

\begin{proposition}[Entropy rate of $P G_S$]\label{prop:entropyPG_S}
    Let $P\in \mathcal{L}(\pi)$ and consider any partition $S,S'$ such that $\mathcal{X} = S \sqcup S'$ where $S \neq \emptyset,\mathcal{X}$. Then we have
    \begin{align*}
        H_\pi(P G_S) = H(\pi) - D_{\mathrm{KL}}^\pi(PG_S \| \Pi) = H(\pi) + U(S) - T(S),
    \end{align*}
    which is a difference of two supermodular (or equivalently submodular) function. In particular, we have
    \begin{equation} \label{eq:argmax_entropy}
        \argmin_S D_{\mathrm{KL}}^\pi(PG_S \| \Pi) = \argmax_S H_\pi(PG_S),
    \end{equation}
    and minimisation of the difference of two submodular functions Proposition \ref{prop:DPGdiffsub} is equivalent to maximising the entropy rate of $PG_S$ with respect to $S$.
\end{proposition}

\begin{proof}
    Equation \eqref{eq:argmax_entropy} follows directly from 
    \begin{align*}
        D_{\mathrm{KL}}^\pi(PG_S \| \Pi) &= \sum_{x,y \in \mathcal{X}} \pi(x) PG_S(x,y)(\log PG_S(x,y) - \log \pi(y))\\
        &= H(\pi) - H_\pi(PG_S).
    \end{align*}
\end{proof}

It is well-known that the Shannon entropy is submodular when viewed as a set function on collections of random variables. For a finite family $(X_i)_{i \in \llbracket n \rrbracket}$, the map $S \mapsto H(X_S)$ is submodular (see, e.g., Section 1.1 of \cite{Krause_Golovin_2014}). 

In Proposition \ref{prop:entropyPG_S}, we instead show that the entropy-related objective associated with the samplers $G_SP$ and $PG_S$ admits a difference-of-submodular (DS) decomposition with respect to the partition variable $S$. While every set function can be written as a difference of submodular functions (\cite{Bilmes_subsup}, Lemma 4), such representations are typically non-constructive. Here, we derive an explicit decomposition that can be exploited algorithmically, connecting the optimisation of $G_SP$ and $PG_S$ to DS optimisation schemes, including the submodular–modular or modular-modular procedures of \cite{Bilmes_subsup}, and more recent difference-of-convex (DC) programming approaches for DS minimisation in \cite{halabi_dc}.

\subsection{An approximate minimisation algorithm for the KL divergence of $PG_S$ from $\Pi$}\label{subsec:MMKL}

In general, a set function $\nu: 2^\mathcal{X} \to \mathbb{R}$ can always be written as a difference of two submodular functions $\nu = u - v$, $u, v: 2^\mathcal{X} \to \mathbb{R}$ (\cite{Bilmes_subsup}, Lemma 4). A standard approach following \cite{iyer2013algorithmsapproximateminimizationdifference} is then to construct modular surrogate functions that locally approximate $u$ and/or $v$. 

Concretely, at a current iterate $A_t \subseteq V$, one selects a modular upper bound $M_t$ of $u$ and a modular lower bound $m_t$ of $v$, both tight at $A_t.$ Minimising the difference $M_t - m_t$ then yields a monotone decrease in the original objective. This produces a descent algorithm towards a local minimum, and is the key idea behind the submodular-supermodular and modular-modular procedures. 

Viewed more broadly, this procedure mirrors the structure of a majorisation-minimisation (MM) algorithm. The procedures above can be seen as a discrete analogue of MM schemes, with modular functions playing the role of majorisers/minorisers. A review of the MM algorithm can be found in \cite{Hunter01022004}.

While Section 2 of \cite{iyer2013algorithmsapproximateminimizationdifference} provides specific constructions of the functions $m_t$ and $M_t$, in the context of our problem, we give a simple explicit choice of $m_t$.

\begin{proposition} \label{prop:U_lower}
    For any fixed $S_t \subseteq \mathcal{X}$ with $0 < \pi(S_t) < 1$, a tight modular lower bound $m_t$ for $U$ is given by
    $$m_t(S) := U(S_t) + \left(\log \frac{\pi(S_t)}{1-\pi(S_t)}\right) \left(\pi(S) - \pi(S_t)\right).$$
    That is, $m_t(S) \leq U(S)$ with equality when $S = S_t$.
\end{proposition}

\begin{proof}
    Recall that 
    $$U(S) = \pi(S)\log \pi(S) + (1-\pi(S))\log (1-\pi(S)) = \Psi(\pi(S)),$$
    where $\Psi: (0,1) \to \mathbb{R},\ \Psi(x) = x\log x + (1-x)\log (1-x).$ Note that $\Psi$ is convex on $(0,1)$, and $\Psi'(x) = \log(x/(1-x))$ satisfies
    $$\Psi(x) \geq \Psi(y) + \Psi'(y)(x-y)$$
    for any $y \in (0,1).$
    Setting $y = \pi(S_t)$ and $x = \pi(S)$ gives us $m_t(S) \leq U(S)$ as desired. Note that $m_t(S)$ is modular in $S$ since it depends on $S$ only through $\pi(S)$.
\end{proof}

We note that this approximation is possible due to the structure of the convex function $\Psi$. In particular, convex functions admit global affine lower bounds given by their tangents, which are sharp at the point of linearisation. This property enables the construction of a tight modular lower bound for $U(S)$ via a first-order approximation in $\pi(S)$.

By contrast, convexity alone cannot be exploited in a similar manner for $T(S)$. A non-affine convex function does not admit a global affine upper bound that is tight at a given point, and hence no sharp modular majoriser of $T(S)$ can be obtained through convexity arguments alone. Constructing a valid global upper bound for $T(S)$ therefore requires additional structure, which in our case is provided by the submodularity of $-T$. The following construction follows from Proposition 5 of \cite{Bilmes_subsup}.

\begin{proposition} \label{prop:T_upper}
    Consider any $P \in \mathcal{L}(\mathcal{X})$ and let 
    $$W_k := \{\Phi(1),\dots,\Phi(k)\}, \qquad k=1,\dots,|\mathcal X|.$$ Choose $\Phi$ to be any permutation of $\mathcal{X}$ such that 
    $$S_t = \{\Phi(1),\dots,\Phi(|S_t|)\}.$$
    Define $M_t : 2^{\mathcal X} \to \mathbb{R}$ by 
    $$M_t(S) :=\sum_{y\in S} M_t(y), \qquad M_t(\Phi(k)) := \begin{cases}
        T(W_1), & \mathrm{if}\ k = 1,\\
        T(W_k) - T(W_{k-1}), & \mathrm{otherwise}.
    \end{cases}$$
    Then $M_t$ is a modular upper bound for $T$, that is,
    $$T(S,P, \pi) = T(S)  \leq M_t(S) \quad \text{for all } S \subseteq \mathcal X,$$
    with equality when $S=S_t$. 
\end{proposition}

Now, one can use either or both of the approximations listed in Propositions \ref{prop:U_lower} and \ref{prop:T_upper} as surrogates to the objective function $T - U.$ Replacing only one of the two terms by its modular approximation yields submodular optimisation problem, for which convex analytic methods is available via the Lovász extension. While such problems admit polynomial-time algorithms, exact optimisation is often computationally expensive and scales poorly in large
state spaces, motivating the use of greedy-type approximation algorithms in practice.

In contrast, replacing both $T$ and $U$ by their respective modular bounds results in a fully modular surrogate objective. Optimisation in this case is simpler and can be carried out exactly and efficiently, at the expense of a coarser local approximation of the original objective. 

\subsection{A coordinate descent algorithm for minimising $D^\pi_{\mathrm{KL}}(G_V P G_S \| \Pi)$ via difference-of-submodular decomposition} \label{sub:doubleMM}

Let $V,S \subseteq \mathcal{X}$. By applying the Pythagorean identity twice (see Corollary 6.1 of \cite{choi2025groupaveragedmarkovchainsmixing}), we see that
\begin{align*}
    D^\pi_{\mathrm{KL}}(P \| \Pi) &= D^\pi_{\mathrm{KL}}(P \| P G_S) + D^\pi_{\mathrm{KL}}(P G_S \| \Pi) \\
                                  &= D^\pi_{\mathrm{KL}}(P \| P G_S) + D^\pi_{\mathrm{KL}}(P G_S \| G_V P G_S) + D^\pi_{\mathrm{KL}}(G_V P G_S \| \Pi).
\end{align*}

We therefore have
\begin{align}\label{eq:argminargmaxGVPGS}
    \argmin_{V,S \subseteq \mathcal{X};~V,S \neq \emptyset, \mathcal{X}} D^\pi_{\mathrm{KL}}(G_V P G_S \| \Pi) = \argmax_{V,S \subseteq \mathcal{X};~V,S \neq \emptyset, \mathcal{X}} D^\pi_{\mathrm{KL}}(P \| P G_S) + D^\pi_{\mathrm{KL}}(P G_S \| G_V P G_S).
\end{align}
Thus, in the remainder of this subsection we focus on the maximisation problem on the right hand side of \eqref{eq:argminargmaxGVPGS}. Note that under the setting of $V \neq S, S'$, in general $G_V P G_S$ is non-reversible, even if $P$ is reversible.

We first work with $D^\pi_{\mathrm{KL}}(P \| P G_S)$ and decompose it as a difference of two submodular functions:
\begin{lemma}\label{lem:decomposePPGS}
    Let $P \in \mathcal{L}(\pi)$ and $S \subseteq \mathcal{X}$. We have
    \begin{align*}
        D^\pi_{\mathrm{KL}}(P \| P G_S) = H(\pi) - H_\pi(P) - T(S) + U(S),
    \end{align*}
    where we recall $T,U$ are two supermodular functions introduced in Proposition \ref{prop:DPGdiffsub}. Thus, the map $\mathcal{X} \supseteq S \mapsto D^\pi_{\mathrm{KL}}(P \| P G_S)$ is a difference of two submodular functions.
\end{lemma}
\begin{proof}
    We have
    \begin{align*}
        D^\pi_{\mathrm{KL}}(P \| P G_S) &= D^\pi_{\mathrm{KL}}(P \| \Pi) - D^\pi_{\mathrm{KL}}(P G_S \| \Pi) \\
        &= H(\pi) - H_\pi(P) - T(S) + U(S),
    \end{align*}
    where the first equality follows from the Pythagorean identity (see Corollary 6.1 of \cite{choi2025groupaveragedmarkovchainsmixing}) and the second equality makes use of Proposition \ref{prop:DPGdiffsub}.
\end{proof}

Next, we handle $D^\pi_{\mathrm{KL}}(P G_S \| G_V P G_S)$. First, we define $\psi:2^{\mathcal{X}} \times 2^{\mathcal{X}} \to \mathbb{R}^+$ to be
\begin{align*}
    \psi(V,S) = \psi(V,S,P,\pi) := - \left(\sum_{x \in V} \pi(x) P(x,S)\right) \log\left(\sum_{k \in V} \pi(k)P(k,S)\right).
\end{align*}
We calculate that
\begin{align}
    D^\pi_{\mathrm{KL}}&(P G_S \| G_V P G_S) \nonumber \\
    &= \sum_{x,y} \pi(x) PG_S(x,y) \log\left(\dfrac{P G_S(x,y)}{G_V P G_S(x,y)} \right) \nonumber \\
    &= \sum_x \pi(x) P(x,S) \log\left(\dfrac{P(x,S)}{G_V P(x, S)}\right) + \sum_x \pi(x) P(x,S') \log\left(\dfrac{P(x,S')}{G_V P(x, S')}\right) \nonumber \\
    &= T(S) - H(\overline{\pi}^V) - \sum_{x \in V} \pi(x) P(x,S) \log\left(\sum_{k \in V} \pi(k)P(k,S)\right) \nonumber \\
    &\quad -  \sum_{x \in V'} \pi(x) P(x,S) \log\left(\sum_{k \in V'} \pi(k)P(k,S)\right) -  \sum_{x \in V} \pi(x) P(x,S') \log\left(\sum_{k \in V} \pi(k)P(k,S')\right) \nonumber \\
    &\quad - \sum_{x \in V'} \pi(x) P(x,S') \log\left(\sum_{k \in V'} \pi(k)P(k,S')\right) \nonumber \\
    &= T(S) - H(\overline{\pi}^V) + \psi(V,S) + \psi(V',S) + \psi(V,S') + \psi(V',S'), \label{eq:decomposeGVPGS}
\end{align}
where we recall that $T$ is introduced in Proposition \ref{prop:DPGdiffsub} and $H(\overline{\pi}^V)$ is the Shannon entropy of the two-point distribution $\overline{\pi}^V := (\pi(V),\pi(V'))$.

\begin{lemma}\label{lem:decompose}
    Let $P \in \mathcal{L}(\pi)$. We have
    \begin{enumerate}
        \item\label{it:HoverlinepiV} The mapping
        \begin{align*}
            \mathcal{X} \supseteq V \mapsto H(\overline{\pi}^V) = -U(V)
        \end{align*}
        is a submodular function, where we recall $U(V)$ is supermodular by Proposition \ref{prop:DPGdiffsub}.

        \item\label{it:psiV} Fix $S \subseteq \mathcal{X}$. The mapping
        \begin{align*}
            \mathcal{X} \supseteq V \mapsto \psi(V,S) = \underbrace{\sum_{y \in S} \pi(y) P(y,V') \log\left(\sum_{k \in V} \pi(k)P(k,S)\right)}_{\text{submodular}} - \underbrace{\pi(S) \log\left(\sum_{k \in V} \pi(k)P(k,S)\right)}_{\text{submodular}}
        \end{align*}
        for $\sum_{k \in V} \pi(k)P(k,S) \neq 0$ is decomposed as a difference of two submodular functions. If  $\sum_{k \in V} \pi(k)P(k,S) = 0$ then $\psi(V,S) = \infty$.

        \item\label{it:psiS} Fix $V \subseteq \mathcal{X}$. The mapping
        \begin{align*}
            \mathcal{X} \supseteq S \mapsto \psi(V,S) = \underbrace{\sum_{x \in V} \pi(x) P(x,S') \log\left(\sum_{k \in V} \pi(k)P(k,S)\right)}_{\text{submodular}} - \underbrace{\pi(V) \log\left(\sum_{k \in V} \pi(k)P(k,S)\right)}_{\text{submodular}}
        \end{align*}
        for $\sum_{k \in V} \pi(k)P(k,S) \neq 0$ is decomposed as a difference of two submodular functions. If  $\sum_{k \in V} \pi(k)P(k,S) = 0$ then $\psi(V,S) = \infty$.
    \end{enumerate}
\end{lemma}

\begin{proof}
    In this proof, we shall repeatedly use the following two properties:
    \begin{enumerate}[(a)]
        \item\label{it:convextransform} By Proposition 6.1 of \cite{bach2013learningsubmodularfunctionsconvex}, a convex (resp.~concave) function composes with a non-negative modular function is supermodular (resp.~submodular).

        \item\label{it:productsuper} By Corollary 2.6.3 of \cite{Topkis_1998}, the product of two non-negative, supermodular and both non-decreasing or both non-increasing set function is supermodular.
    \end{enumerate}

    We first show item \eqref{it:HoverlinepiV}, which follows from the fact that $U(V)$ is supermodular by Proposition \ref{prop:DPGdiffsub}.

    Next, we prove item \eqref{it:psiV}. First, by reversibility of $P$ we compute that
    \begin{align*}
        \psi(V,S) &= - \left(\sum_{y \in S} \pi(y) P(y,V)\right) \log\left(\sum_{k \in V} \pi(k)P(k,S)\right) \\
        &= - \left(\sum_{y \in S} \pi(y) (1-P(y,V'))\right) \log\left(\sum_{k \in V} \pi(k)P(k,S)\right) \\
        &= \sum_{y \in S} \pi(y) P(y,V') \log\left(\sum_{k \in V} \pi(k)P(k,S)\right) - \pi(S) \log\left(\sum_{k \in V} \pi(k)P(k,S)\right)
    \end{align*}
    By the previous paragraph we note that $\log\left(\sum_{k \in V} \pi(k)P(k,S)\right)$ is submodular in $V$. We also see that $-P(y,V') \log\left(\sum_{k \in V} \pi(k)P(k,S)\right)$ is supermodular as it is a product of a non-increasing, non-negative modular function $P(y,V')$ and non-increasing, non-negative supermodular function $-\log\left(\sum_{k \in V} \pi(k)P(k,S)\right)$. This item is shown as summing over $y \in S$ preserves submodularity.

    Finally, we prove item \eqref{it:psiS}. We see that
    \begin{align*}
        \psi(V,S) &= - \left(\sum_{x \in V} \pi(x) (1-P(x,S'))\right) \log\left(\sum_{k \in V} \pi(k)P(k,S)\right) \\
        &= \sum_{x \in V} \pi(x) P(x,S') \log\left(\sum_{k \in V} \pi(k)P(k,S)\right) - \pi(V) \log\left(\sum_{k \in V} \pi(k)P(k,S)\right).
    \end{align*}
    Similarly to the previous two paragraphs, we have $\log\left(\sum_{k \in V} \pi(k)P(k,S)\right)$ is submodular in $S$. We also note that $-P(x,S') \log\left(\sum_{k \in V} \pi(k)P(k,S)\right)$ is supermodular as it is a product of a non-increasing, non-negative modular function $P(x,S')$ and non-increasing, non-negative supermodular function $-\log\left(\sum_{k \in V} \pi(k)P(k,S)\right)$. This item is shown as summing over $x \in V$ preserves submodularity.
\end{proof}

By collecting Lemma \ref{lem:decomposePPGS}, \eqref{eq:decomposeGVPGS} and Lemma \ref{lem:decompose}, we see that the maximisation problem in \eqref{eq:argminargmaxGVPGS} boils down to
\begin{align*}
    &\argmax_{V,S \subseteq \mathcal{X};~V,S \neq \emptyset, \mathcal{X}} D^\pi_{\mathrm{KL}}(P \| P G_S) + D^\pi_{\mathrm{KL}}(P G_S \| G_V P G_S) \\
    &= \argmax_{V,S \subseteq \mathcal{X};~V,S \neq \emptyset, \mathcal{X}} U(S) - H(\overline{\pi}^V) + \psi(V,S) + \psi(V',S) + \psi(V,S') + \psi(V',S') \\
    &= \argmin_{V,S \subseteq \mathcal{X};~V,S \neq \emptyset, \mathcal{X}} -U(S) + H(\overline{\pi}^V) - \psi(V,S) - \psi(V',S) - \psi(V,S') - \psi(V',S') =: \varphi(V,S).
\end{align*}
Note that $\varphi(V,S)$ is a difference-of-submodular function.

We now propose a coordinate descent type algorithm to minimize $\varphi(V,S)$. Suppose that the algorithm is initialized at $V^0 \subseteq \mathcal{X}$. By Lemma \ref{lem:decompose}, the mapping $S \mapsto \varphi(V^0,S)$ is a difference-of-submodular function, thus one can invoke existing algorithms for (approximately) minimizing 
$$S^0 \in \argmin_{S \subseteq \mathcal{X};~S \neq \emptyset, \mathcal{X}} \varphi(V^0,S).$$
Using Lemma \ref{lem:decompose} again, the mapping $V \mapsto \varphi(V,S^0)$ is a difference-of-submodular function, and we (approximately) minimize
$$V^1 \in \argmin_{V \subseteq \mathcal{X};~V \neq \emptyset, \mathcal{X}} \varphi(V,S^0).$$
Proceeding iteratively, we alternately minimize the two coordinates. At iteration $l \in \mathbb{N}$, we (approximately) minimize
\begin{align*}
    V^l \in \argmin_{V \subseteq \mathcal{X};~V \neq \emptyset, \mathcal{X}} \varphi(V,S^{l-1}), \quad S^l \in \argmin_{S \subseteq \mathcal{X};~S \neq \emptyset, \mathcal{X}} \varphi(V^l,S).
\end{align*}
This yields a monotonically non-increasing sequence:
\begin{align*}
    \varphi(V^0,S^0) \geq \varphi(V^1,S^1) \geq \ldots \geq \varphi(V^l,S^l).
\end{align*}

\subsection{Submodular optimisation, minimisation of Frobenius norm and a MM type algorithm} \label{sub:frobsubmod}

In this subsection, we reveal the submodularity structure arising in the Frobenius norm between $G_S P$ or $G_S P G_S$ and $\Pi$. This allows us to take advantage of such structure to design optimisation algorithms to tune the choice of $S$.

First, we present a useful lemma:
\begin{lemma}\label{lem:Qxy}
    Assume that $P \in \mathcal{L}(\pi)$ and let $S \subseteq \mathcal{X}$. We have
    \begin{align}
        \sum_{x \in S} \sum_{y \in S'} \pi(x) P(x,y) &= \pi(S) - \sum_{x \in S} \sum_{y \in S} \pi(x) P(x,y) \label{eq:sumSS}\\
                                                     &= \pi(S') - \sum_{x \in S'} \sum_{y \in S'} \pi(x) P(x,y), \label{eq:sumS'S'}
    \end{align}
    where $\sum_{x \in \emptyset} := 0$. Consequently, the mappings
    \begin{align*}
        \mathcal{X} \supseteq S &\mapsto \sum_{x \in S} \sum_{y \in S} \pi(x) P(x,y), \quad 
        \mathcal{X} \supseteq S \mapsto \sum_{x \in S'} \sum_{y \in S'} \pi(x) P(x,y)
    \end{align*}
    are supermodular.
\end{lemma}

\begin{proof}
    When $S = \emptyset, \mathcal{X}$, the result trivially holds. Thus, in the remainder of the proof we assume that $S \neq \emptyset, \mathcal{X}$.

    To see the first equality \eqref{eq:sumSS} in the lemma statement, we see that
    \begin{align*}
        \sum_{x \in S} \sum_{y \in S'} \pi(x) P(x,y) &= \sum_{x \in S} \pi(x) (1-P(x,S)) = \pi(S) - \sum_{x \in S} \sum_{y \in S} \pi(x) P(x,y).
    \end{align*}
    To prove the second equality \eqref{eq:sumS'S'} in the statement, we note that, by the reversibility of $P$, 
    \begin{align*}
        \sum_{x \in S} \sum_{y \in S'} \pi(x) P(x,y) &= \sum_{y \in S'}\sum_{x \in S} \pi(y)P(y,x) = \pi(S') - \sum_{x \in S'} \sum_{y \in S'} \pi(x) P(x,y).
    \end{align*}
    where the last equality follows from \eqref{eq:sumSS} by replacing $S$ with $S'$.

    $\sum_{x \in S'} \sum_{y \in S'} \pi(x) P(x,y)$ is the complement function of $\sum_{x \in S} \sum_{y \in S} \pi(x) P(x,y)$. Since complement preserves supermodularity, it thus suffices for us to show that $\sum_{x \in S} \sum_{y \in S} \pi(x) P(x,y)$ is supermodular in $S$, which is indeed the case since by \eqref{eq:sumSS} it is the difference of a modular function $\pi(S)$ and a submodular function $\sum_{x \in S} \sum_{y \in S'} \pi(x) P(x,y)$ (a non-negative cut function is submodular, see e.g. \cite[Section $6.2$]{bach2013learningsubmodularfunctionsconvex}).
\end{proof}

Next, we recognize two set functions as supermodular and one set function as submodular:
\begin{lemma}
    Assume that $P \in \mathcal{L}(\pi)$. The following two mappings
    \begin{align*}
        \{S \subseteq \mathcal{X}; 0 < \pi(S) < 1\} \ni S &\mapsto \dfrac{1}{\pi(S)} \sum_{x \in S'} \sum_{y \in S'} \pi(x) P(x,y) \\
        \{S \subseteq \mathcal{X}; 0 < \pi(S) < 1\} \ni S &\mapsto \dfrac{1}{\pi(S')} \sum_{x \in S} \sum_{y \in S} \pi(x) P(x,y)
    \end{align*}
    are supermodular while the following mapping
    \begin{align*}
        \{S \subseteq \mathcal{X}; 0 < \pi(S) < 1\} \ni S &\mapsto 3 - \dfrac{1}{\pi(S)\pi(S')}
    \end{align*}
    is submodular.
\end{lemma}

\begin{proof}
    We again make use of item \eqref{it:convextransform} and \eqref{it:productsuper} repeatedly. 
    
    Precisely, since the mapping $(0,1) \ni t \mapsto 1/t$ is convex, $1/\pi(S)$ (resp.~$1/\pi(S')$) is thus supermodular by item \eqref{it:convextransform} and furthermore non-increasing (resp.~non-decreasing). Therefore, $\frac{1}{\pi(S)} \sum_{x \in S'} \sum_{y \in S'} \pi(x) P(x,y)$ can be seen as a product of two non-increasing, non-negative, supermodular functions by Lemma \ref{lem:Qxy}, and hence supermodular by item \eqref{it:productsuper}. Similarly, $\frac{1}{\pi(S')} \sum_{x \in S} \sum_{y \in S} \pi(x) P(x,y)$ can be seen as a product of two non-decreasing, non-negative, supermodular functions by Lemma \ref{lem:Qxy}, and hence supermodular by item \eqref{it:productsuper}.

    We also note that the mapping $(0,1) \ni t \mapsto 3 - \frac{1}{t(1-t)}$ is concave, and hence by item \eqref{it:convextransform}, $3 - \dfrac{1}{\pi(S)\pi(S')}$ is submodular.
\end{proof}

Taking advantages of the previous two lemmas, the main result of this subsection decomposes the Frobenius norm of $G_S P G_S$ or $G_S P$ to $\Pi$ as a difference of two submodular functions:
\begin{proposition}\label{prop:GPGFdecompose}
    Assume that $P \in \mathcal{L}(\pi)$ and $S \subseteq \mathcal{X}$ with $S \neq \emptyset,\mathcal{X}$.
    \begin{enumerate}
        \item\label{it:GPGdecompose} Assume further that $P$ is positive-semidefinite. We decompose
        \begin{align*}
            \|G_SPG_S - \Pi\|_{F, \pi} &= \underbrace{3 - \dfrac{1}{\pi(S)\pi(S')}}_{\text{submodular}} + \underbrace{\dfrac{1}{\pi(S)} \sum_{x \in S'} \sum_{y \in S'} \pi(x) P(x,y)}_{\text{supermodular}} \\
            &\quad + \underbrace{\dfrac{1}{\pi(S')} \sum_{x \in S} \sum_{y \in S} \pi(x) P(x,y)}_{\text{supermodular}}
        \end{align*}
        as a difference of two submodular functions.

        \item\label{it:GPdecompose} We decompose
        \begin{align*}
            \|G_SP - \Pi\|_{F, \pi}^2 &= \underbrace{3 - \dfrac{1}{\pi(S)\pi(S')}}_{\text{submodular}} + \underbrace{\dfrac{1}{\pi(S)} \sum_{x \in S'} \sum_{y \in S'} \pi(x) P^2(x,y)}_{\text{supermodular}} \\
            &\quad + \underbrace{\dfrac{1}{\pi(S')} \sum_{x \in S} \sum_{y \in S} \pi(x) P^2(x,y)}_{\text{supermodular}}
        \end{align*}
        as a difference of two submodular functions.
    \end{enumerate}
\end{proposition}

\begin{proof}
    We first prove item \eqref{it:GPGdecompose}. By Proposition \ref{prop:GSPGSlPi}, we compute that
    \begin{align*}
        \|G_SPG_S - \Pi\|_{F, \pi} &= 1 - \dfrac{1}{\pi(S)\pi(S')} \sum_{\substack{x \in S\\ y \in S'}} \pi(x) P(x,y) \\
                                   &= 1 -  \dfrac{1}{\pi(S)\pi(S')} \left(\pi(S') - \sum_{x \in S'} \sum_{y \in S'} \pi(x) P(x,y)\right) \\
                                   &= 1 - \dfrac{1}{\pi(S)} + \dfrac{1}{\pi(S)} \sum_{x \in S'} \sum_{y \in S'} \pi(x) P(x,y) \\
                                   &\quad + \dfrac{1}{\pi(S')} \sum_{x \in S'} \sum_{y \in S'} \pi(x) P(x,y) \\
                                   &= 1 - \dfrac{1}{\pi(S)} + \dfrac{1}{\pi(S)} \sum_{x \in S'} \sum_{y \in S'} \pi(x) P(x,y) \\
                                   &\quad + \dfrac{1}{\pi(S')} \left(\pi(S') - \pi(S) + \sum_{x \in S} \sum_{y \in S} \pi(x) P(x,y)\right) \\
                                   &= 3 - \dfrac{1}{\pi(S)\pi(S')} + \dfrac{1}{\pi(S)} \sum_{x \in S'} \sum_{y \in S'} \pi(x) P(x,y) \\
            &\quad + \dfrac{1}{\pi(S')} \sum_{x \in S} \sum_{y \in S} \pi(x) P(x,y),
    \end{align*}
    where we use Lemma \ref{lem:Qxy} in the second and the fourth equality.

    Next, we prove item \eqref{it:GPdecompose}. By Proposition \ref{prop:GP argmin}, we note that
    \begin{align*}
        \|G_SP - \Pi\|_{F, \pi}^2 &= 1 - \dfrac{1}{\pi(S)\pi(S')} \sum_{\substack{x \in S\\ y \in S'}} \pi(x) P^2(x,y),
    \end{align*}
    and the desired result follows from the exact same proof of item \eqref{it:GPGdecompose} with $P$ being replaced by $P^2$.
\end{proof}

Similar to the earlier Subsection \ref{subsec:MMKL} where we design majorisation-minimisation (MM) type algorithms to minimize the KL divergence, one can make use of the difference-of-submodular structure as presented in Proposition \ref{prop:GPGFdecompose} to design MM type algorithms to minimise the Frobenius norm.

Specifically, since $(0,1) \ni t \mapsto 1/(t(1-t))$ is convex, by the supporting hyperplane theorem for convex differentiable function we have, for $S,S^0 \subseteq \mathcal{X}$ with $S, S^0 \neq \emptyset,\mathcal{X}$, 
\begin{align*}
    \dfrac{1}{\pi(S)\pi(S')} \geq \dfrac{1}{\pi(S^0)\pi(S^{0'})} + \dfrac{2\pi(S^0) - 1}{\pi(S^0)^2 (1 - \pi(S^0))^2}(\pi(S) - \pi(S^0)),
\end{align*}
in which this lower bound is clearly modular in $S$. Using the above inequality in Proposition \ref{prop:GPGFdecompose} we see that for reversible $P$, we obtain a majorisation function $\zeta(S;S^0)$ which is supermodular in $S$:
\begin{align*}
            \|G_SP - \Pi\|_{F, \pi}^2 &\leq 3 - \dfrac{1}{\pi(S^0)\pi(S^{0'})} + \dfrac{1 - 2\pi(S^0)}{\pi(S^0)^2 (1 - \pi(S^0))^2}(\pi(S) - \pi(S^0)) \\
            &\quad + \dfrac{1}{\pi(S)} \sum_{x \in S'} \sum_{y \in S'} \pi(x) P^2(x,y) 
            + \dfrac{1}{\pi(S')} \sum_{x \in S} \sum_{y \in S} \pi(x) P^2(x,y) \\
            &=: \zeta(S;S^0) = \zeta(S,P^2,\pi;S^0)
\end{align*}

Now, we propose a MM type algorithm to minimise $\|G_SP - \Pi\|_{F, \pi}^2$. Suppose that the algorithm is initialised at $S^0 \subseteq \mathcal{X}$ with $S^0 \neq \emptyset,\mathcal{X}$. We then (approximately) minimise the supermodular majorisation function
\begin{align*}
    S^1 \in \argmin_{S \subseteq \mathcal{X};~ S \neq \emptyset, \mathcal{X}} \zeta(S;S^0).
\end{align*}
This supermodular minimisation can be done for instance using the approximation algorithms presented in \cite{uriel_2011}. This yields
\begin{align*}
    \|G_{S^0}P - \Pi\|_{F, \pi}^2 = \zeta(S^0;S^0) \geq \zeta(S^1;S^0) \geq \|G_{S^1}P - \Pi\|_{F, \pi}^2 = \zeta(S^1;S^1).
\end{align*}
We then iterate the above procedure. For $l \in \mathbb{N}$, we (approximately) minimise the supermodular majorisation function
\begin{align*}
    S^l \in \argmin_{S \subseteq \mathcal{X};~ S \neq \emptyset, \mathcal{X}} \zeta(S;S^{l-1}).
\end{align*}
This gives a non-increasing sequence of Frobenius norm:
\begin{align*}
    \|G_{S^0}P - \Pi\|_{F, \pi}^2 \geq \|G_{S^1}P - \Pi\|_{F, \pi}^2 \geq \ldots \geq \|G_{S^l}P - \Pi\|_{F, \pi}^2.
\end{align*}

Similar MM type algorithms can also be developed for minimising $\|G_{S}P G_S - \Pi\|_{F, \pi}$.

\section{Positive definite samplers $P$} \label{sec:posdef}
In this section, we investigate characteristics of the samplers $GP$ and $GPG$ when $P$ is taken to be a positive definite sampler with respect to $\langle \cdot,\cdot \rangle_\pi.$ Recall the Sylvester's equation, where given any $A,B,C \in \mathbb{R}^{n\times n}$, we have a linear matrix equation in $X \in \mathbb{R}^{n\times n}$ given by 
$$AX + XB = C.$$

Sylvester's theorem states that for any given $A, B$, the equation $AX + XB = C$ has a unique solution for $X$ for any $C$ if and only if $A$ and $-B$ do not share any eigenvalues. A review of the results can be found in \cite[Theorem 2.4.4.1]{Horn_Johnson_1991}.

\begin{proposition}
    Suppose $P \in \mathcal{S}(\pi)$ and both $P, G \neq \Pi$. If $GP = \Pi$, $P$ and $-P^*$ must share at least one common eigenvalue.   
\end{proposition}

\begin{proof}
    Consider the Sylvester's equation 
    $$XP + P^*X = 2\Pi.$$
    Since $GP = \Pi$, $X$ has two solutions $G$ and $\Pi$. It then follows from Sylvester's theorem that $P$ and $-P^*$ share at least one common eigenvalue.
\end{proof}

\begin{corollary}
    For any $\pi$-reversible, ergodic and positive-definite $P$, if $G, P \neq \Pi$, then $GP \neq \Pi.$
\end{corollary}

\begin{proof}
    This follows from the preceding proposition, noting that a positive definite $P$ must have eigenvalues in $(0,1]$. It is thus impossible for $P$ and $-P^* = -P$ to share any eigenvalue. 
\end{proof}

A similar observation can be shown for $GPG$.

\begin{proposition}
    Suppose $P \in \mathcal{S}(\pi)$ and $P, G \neq \Pi$. If $GPG = \Pi$, $GP$ and $-P^*G$ must share at least one common eigenvalue. Further, $PG$ and $-GP^*$ must also share at least one common eigenvalue. 
\end{proposition}

\begin{proof}
    Consider the two equations 
    \begin{align*}
        XPG + GP^*X &= 2\Pi\\
        XP^*G + GPX &= 2\Pi.
    \end{align*}
    For both, $X$ has solutions $X = G$ and $X = \Pi$. The conclusion holds by applying Sylvester's theorem.
\end{proof}

However, unlike the case for $GP$, it is possible for $GP$ and $-P^*G$ to share the common eigenvalue $0$. We thus take a different approach in showing how $GPG \neq \Pi$ when $P$ is positive definite. 

\begin{proposition}
    For any $\pi$-reversible, ergodic and positive definite $P$, if $G, P \neq \Pi$ then $GPG \neq \Pi.$
\end{proposition}

\begin{proof}
    Define $\mathrm{Im}(G)$ to be the image of $G$ and $\mathrm{ker}(G)$ be its kernel. 
    
    Since $G \neq \Pi$, $\mathrm{dim}(\mathrm{Im}(G)) \geq 2$, and so there exists a non-zero $h \in \mathrm{Im}(G)$ with mean zero (i.e. $\langle h, \mathbf{1}\rangle_\pi = 0$).

    Suppose $GPG = \Pi$. It follows that 
    $$GPG\ h = GP\ h = \Pi h = 0,$$
    where the last equality holds since $h$ is mean 0. For example, one such $h$ would be 
    $$h(x) = \begin{cases}
        \pi(\mathcal{O}_2) & \mathrm{for}\ x \in \mathcal{O}_1,\\
        -\pi(\mathcal{O}_1) & \mathrm{for}\ x \in \mathcal{O}_2,\\
        0 & \mathrm{otherwise}.
    \end{cases}$$

    Since $GPh = 0$, $Ph \in \mathrm{ker}(G)$, and so it must follow that 
    $$\langle h, Ph \rangle_\pi = 0,$$
    which contradicts the fact that $P$ is positive-definite given that $h$ is non-zero.
\end{proof}

\section{Numerical experiments} \label{sec:num}
In this section, we present numerical results regarding the performance of the various samplers discussed. The source codes for all plots and results can be found in \url{https://github.com/ryan-limjy/submodular}.

We briefly recall the Curie-Weiss model that we use as a benchmark to test the convergence of various samplers. Consider the discrete $d$-dimensional state space 
$$\mathcal{X} = \{-1, +1\}^d,$$
and define the Hamiltonian function for $x = (x^1, \dots, x^d) \in \mathcal{X}$ be  
$$\mathcal{H}(x) = - \sum_{i,j=1}^d \frac{1}{2^{|j-i|}} x^i x^j - h \sum_{i=1}^d x^i.$$
In this context, the interaction coefficient is given by $\frac{1}{2^{|j-i|}}$ and the external magnetic field is $h \in \mathbb{R}$. An in-depth discussion of the model can be found in \cite{Bovier_denHollander_2015}.

The stationary distribution $\pi$ is then the Gibbs distribution at temperature $T > 0.$ That is, 
$$\pi(x) = \frac{e^{-\frac{1}{T} \mathcal{H}(x)}}{\sum_{z\in \mathcal{X}}e^{-\frac{1}{T} \mathcal{H}(z)}}.$$

We then consider our baseline sampler $P$ to be a Glauber dynamics with a simple random walk proposal. Let 
\begin{equation} \label{eq:glauber}
    P(x,y) = \begin{cases}
    \frac{1}{d} e^{-\frac{1}{T}(\mathcal{H}(y)-\mathcal{H}(x))_+}, & \mathrm{for}\ y = (x^1, \dots, -x^i, \dots, x^d), i \in \llbracket d \rrbracket,\\
    1-\sum_{y \neq x} P(x,y), & \mathrm{if}\ x = y,\\
    0, & \mathrm{otherwise,}
\end{cases}
\end{equation}

where for $m \in \mathbb{R},$ $m_+ = \max\{m,0\}$. This sampler is equivalent to uniformly picking one of the $d$ coordinates, flipping it to the opposite sign and performing an acceptance-rejection stage.

\subsection{Total variation distance across different optimal choices of $G_S$}
This subsection presents empirical results regarding the total variation (TV) distance of various samplers. We define the total variation distance between any two probability distributions $\mu, \nu$ on $\mathcal{X}$ to be 
$$\| \mu - \nu \|_{TV} := \frac{1}{2} \sum_{x\in \mathcal{X}} |\mu(x)-\nu(x)|,$$
and we denote
$$\mathbb{N} \ni t \mapsto \max_{x\in \mathcal{X}} \|P^t(x,\cdot) - \pi \|_{TV}$$
to be the worst-case TV distance of the sampler $P$. In subsequent figures, we shall be comparing the worst-case TV distance between various samplers such as $P, G_S P, G_S P G_S$.

Below, we plot the worst-case TV distance against $t$ for the Curie-Weiss model with $d = 4$ particles. Four models are considered, pairing up high temperature $T = 15$ and low temperature $T = 2$, with external fields $h = 0$ and $h = 2$.

The base sampler $P$ chosen is the Glauber dynamics as given in \eqref{eq:glauber}. The partition $\mathcal{X} = S \sqcup S'$ is chosen using the various metrics discussed:
\begin{enumerate}
    \item A random cut $S$
    \item $S \in \argmin_{S;~S \neq \mathcal{X}, \emptyset} \|G_SPG_S - \Pi\|_{F, \pi}^2$, obtained via brute-force search. This method is called ``Min Frobenius (GPG)" in subsequent figures.
    \item $S \in \argmin_{S;~S \neq \mathcal{X}, \emptyset} D_{\mathrm{KL}}^\pi (G_SPG_S \| \Pi )$, obtained via brute-force search. This method is called ``Min KL (GPG)" in subsequent figures.
    \item $S$ minimising $D_{\mathrm{KL}}^\pi (G_S \| \Pi )$ as suggested in Proposition 7.1 of \cite{GPG}. This method is called ``Min $\pi$ singleton" in subsequent figures.
\end{enumerate}
Note that the above optimisation is performed under brute-force search, and the results in Figures \ref{fig:tvplot} and \ref{fig:cut_by_magnetisation} utilises exact solutions rather than an approximation. 

In Figure \ref{fig:cut_by_magnetisation}, we plot the stationary $\pi-$mass contained in each cut $S$ (blue) and its complement $S'$ (orange), aggregated by magnetisation levels $m_x = \sum_i x^i.$ We show the four different optimal choices of $S$ as described above.

The result in Figure \ref{fig:tvplot} reveals that all three optimisation criteria: $\|G_SPG_S - \Pi\|_{F, \pi}^2$, $D_{\mathrm{KL}}^\pi (G_SPG_S \| \Pi )$ and $D_{\mathrm{KL}}^\pi (G_S \| \Pi )$,
yield nearly indistinguishable worst-case TV curves. Further, the optimal $S$ chosen by minimising $\|G_SPG_S - \Pi\|_{F, \pi}^2$ and $D_{\mathrm{KL}}^\pi (G_SPG_S \| \Pi )$ were identical in all but the case where $T = 2, h = 0$, where multiple cuts are nearly optimal, as shown in Figure \ref{fig:cut_by_magnetisation}. 

At higher temperature, the choice of the cut becomes less critical, as even a randomly chosen partition yields a substantial improvement in mixing. This reflects the weaker energy barriers and increased homogeneity of the stationary distribution of the Curie-Weiss model.

We present similar experimental results for $G_SP$, with the same setup as the one before. Precisely, the base sampler $P$ is the Glauber dynamics as given in \eqref{eq:glauber}, while the subset $S$ in the partition $\mathcal{X} = S \sqcup S'$ is chosen using the metrics discussed earlier:
\begin{enumerate}
    \item A random cut $S$
    \item $S \in \argmin_{S;~S \neq \mathcal{X}, \emptyset} \|G_SP - \Pi\|_{F, \pi}^2$, obtained via brute-force search. This method is called ``Min Frobenius (GP)" in subsequent figures.
    \item $S \in \argmin_{S;~S \neq \mathcal{X}, \emptyset} D_{\mathrm{KL}}^\pi (G_SP \| \Pi )$, obtained via brute-force search. This method is called ``Min KL (GP)" in subsequent figures.
    \item $S$ minimising $D_{\mathrm{KL}}^\pi (G_S \| \Pi )$ as suggested in Proposition 7.1 of \cite{GPG}. This method is called ``Min $\pi$ singleton" in subsequent figures.
\end{enumerate}
We remark that the findings presented in Figures \ref{fig:tvplotGS} and \ref{fig:cut_by_magnetisationGP} yield visually similar results when compared with Figures \ref{fig:tvplot} and \ref{fig:cut_by_magnetisation}.

An interesting observation from both Figure \ref{fig:tvplot} and \ref{fig:tvplotGS} is that a random cut $S$ is able to improve the convergence of $G_S P$ and $G_S P G_S$ when compared with $P$ for both high and low temperature and with or without external field $h$, based on the worst-case TV distance.

\subsection{Numerical experiments of various approximation algorithms}

We evaluate the practical performance of the optimisation procedures proposed in this paper, including modular-modular and MM type algorithms, using the Curie-Weiss model with Glauber dynamics $P$ as a representative testbed. Our experiments assess both the quality of the resulting approximate cuts $S$ relative to the true optimum, and the convergence behaviour of each method with respect to the target objective.

\subsubsection{Approximation in Sections \ref{sec:GP_k=2} and \ref{sec:GPG_k=2}}

First, we evaluate the optimality of the approximations given in Proposition \ref{prop:1/2-approx} and \ref{prop:1/2-approxGPG}, where for positive-semidefinite $P \in \mathcal{L}(\pi)$, we approximate, for $|\mathcal{X}| \geq 2$, that
\begin{align*}
    \argmin_{S \neq \emptyset, \mathcal{X}} \|G_SP - \Pi\|_{F, \pi}^2 &\approx \argmin_{x \in \mathcal{X}} P^2(x,x),\\
    \argmin_{S \neq \emptyset, \mathcal{X}} \|G_SPG_S - \Pi\|_{F, \pi} &\approx \argmin_{x \in \mathcal{X}} P(x,x).
\end{align*}

Below, we compare the objective and the optimal set $S$ against the four Curie-Weiss models used in the previous section. Note that the sampler used is now a lazified sampler (i.e. $1/2(P+I)$), where $P$ is the usual Glauber dynamics. The results are shown below in Figure \ref{fig:1/2approx}.

The results suggest that the diagonal approximation performs best when the temperature is low and the external field is non-zero. In this regime, the stationary distribution is strongly skewed, and the optimal cut for both $G_SP$ and $G_SPG_S$ typically reduces to a singleton, which aligns well the proposed approximation.

At higher temperatures, the stationary distribution becomes closer to uniform, and the optimal cut is correspondingly less lopsided. Also when $h=0$, symmetry in the model induces multiple metastable modes, and the optimal cut is no longer a singleton but instead a more balanced partition separating these modes. 

Overall, these observations indicate that the approximation is most effective when the stationary distribution is strongly concentrated on a small number of states, and less suitable in regimes where mixing is governed by collective or symmetric metastability. 

\subsubsection{Approximation in Section \ref{subsec:MMKL}} \label{}
Next, we evaluate the effectiveness of the minimisation algorithm introduced in Section~\ref{subsec:MMKL}. The modular lower and upper bounds were constructed according to Propositions~\ref{prop:U_lower} and~\ref{prop:T_upper}, and experiments were conducted on the Curie–Weiss model with $d=4$.

Note that since the objective given by Proposition \ref{prop:DPGdiffsub} is symmetric up to complements, it suffices for us to consider only the subsets
$$S \subset \mathcal{X}, \{-1\}^d \in S.$$
This reduces the effective number of configurations to $2^{2^d-1}-1 = 2^{15}-1$ subsets after removing the trivial set $S = \mathcal{X}$. 

We test four parameter regimes corresponding to $(T,h) \in \{2,5\} \times \{0,2\}$. For each setting, a brute-force enumeration over the reduced state space was performed to identify all globally optimal subsets (up to complement). This provides the complete set of global minimisers and serves as a benchmark against which the MM algorithm is evaluated.

Since the MM procedure constructs modular bounds using a random permutation, as well as a random starting subset at each iteration, the algorithm is inherently stochastic. To assess robustness, we therefore perform multiple independent runs with different random seeds and record the proportion of runs that converge to a globally optimal solution.

Table \ref{tab:secMMKL} below compares the empirical hit-rate of the MM algorithm with the uniform baseline probability of randomly selecting an optimal subset from the search space. This comparison allows us to distinguish genuine algorithmic effectiveness from success due to chance, thereby quantifying the effectiveness of the MM algorithm.

\begin{table}[ht]
    \centering
    \small
    \begin{tabular}{|p{3.8cm}|c|c|c|c|}
        \hline
        \textbf{Model parameters} 
        & $T=2, h=0$ 
        & $T=2, h=2$ 
        & $T=5, h=0$ 
        & $T=5, h=2$ \\
        \hline
        
        Convergence rate (out of 1000 runs) 
        & 0.145 & 0.535 & 0.002 & 0.011 \\
        \hline
        
        No.\ of optimal subsets 
        & 2 & 2 & 2 & 2 \\
        \hline
        
        Probability of uniformly picking optimal subset 
        & $6.1\times10^{-5}$ 
        & $6.1\times10^{-5}$ 
        & $6.1\times10^{-5}$ 
        & $6.1\times10^{-5}$ \\
        \hline
    \end{tabular}
    \caption{MM convergence results for Section \ref{subsec:MMKL}. ``Convergence rate (out of $1000$ runs)" here refers to the proportion of runs of the MM algorithm, out of 1000 runs, that converge to the optimal $S$ obtained via brute-force.}
    \label{tab:secMMKL}
\end{table}

The results shown in Table \ref{tab:secMMKL} reveal that the algorithm works best for the case where $T = h = 2$, likely due to the strongly skewed landscape that allows for smooth descent towards the global minima. 

As the temperature increases and the distribution becomes flatter, the MM iterates are more prone to becoming trapped in local minima, leading to a noticeable deterioration in hit rate.

Overall, Table \ref{tab:secMMKL} reveals that across all models, the MM algorithm proposed has a higher empirical chance of selecting the optimal subset over random uniform selection. 

\subsubsection{Approximation in Section \ref{sub:doubleMM}}
Finally, we evaluate the coordinate descent algorithm introduced in Section~\ref{sub:doubleMM}. Recall that the objective $D^\pi_{\mathrm{KL}}(G_V P G_S \| \Pi)$ admits a decomposition as a difference of two submodular functions. By Lemma~\ref{lem:decompose}, this decomposition separates over the variables $S$ and $V$, allowing us to optimize each argument while holding the other fixed.

We therefore implemented the coordinate descent scheme, using the MM algorithm described in \cite{iyer2013algorithmsapproximateminimizationdifference} for the minimisation of one variable (in $S$ or $V$) in each descent step. We repeat the algorithm until no further decrease of the objective is observed. 

For this experiment, we restricted the Curie-Weiss model at $d = 3$, so that it remains computationally feasible to solve the minimisation via brute-force. The parameter settings considered are again $(T,h) \in \{2,5\} \times \{0,2\}$, and just like the previous experiment, we only consider subsets up to complement. 

For each iteration, we randomly initialise the starting sets $V_0$ and $S_0$, and we record the empirical hit-rate of the algorithm. Note that in this case, after restricting the trivial set $S = \mathcal{X}$ and deduplication of the complement, the total number of pairs $S \times V$ is $(2^7-1)^2 = 16129$.

\begin{table}
    \centering
    \small
    \begin{tabular}{|p{3.8cm}|c|c|c|c|}
        \hline
        \textbf{Model parameters} 
        & $T=2, h=0$ 
        & $T=2, h=2$ 
        & $T=5, h=0$ 
        & $T=5, h=2$ \\
        \hline
        
        Convergence rate (out of 100 runs) 
        & 0.39 & 0.13 & 0.34 & 0.10 \\
        \hline
        
        No.\ of optimal subsets 
        & 200 & 8 & 200 & 4 \\
        \hline
        
        Probability of uniformly picking optimal subset 
        & 0.0124 
        & $4.96\times10^{-4}$ 
        & 0.0124  
        & $2.48\times10^{-4}$ \\
        \hline
    \end{tabular}
    \caption{Coordinate descent convergence results for Section \ref{sub:doubleMM}. ``Convergence rate (out of $100$ runs)" here refers to the proportion of runs of the coordinate descent algorithm, out of 100 runs, that converge to the optimal $S,V$ obtained via brute-force.}
    \label{tab:doubleMM}
\end{table}

The results in Table~\ref{tab:doubleMM} show that the coordinate descent algorithm substantially outperforms uniform random selection of $(S,V)$. While the absolute convergence rate is higher in the zero-field case $(h=0)$, this setting also admits many more global optima. In contrast, when $h=2$, the number of optimal pairs is significantly smaller, yet the algorithm still converges with non-negligible probability. Relative to the uniform baseline, the improvement factor is markedly larger when $h=2$, suggesting that the algorithm is more effective when working on skewed distributions.

We note that although each MM step is computationally inexpensive, the alternating descent may require many successive MM updates before reaching a local minimum. Consequently, even for $d=3$, the overall procedure can become computationally demanding depending on initialisation.

\subsubsection{Practical limitations}
We now discuss several practical limitations encountered in implementing the proposed algorithm in Section \ref{sub:frobsubmod}. In Proposition~\ref{prop:GPGFdecompose}, the decomposition involves the term
$$\frac{1}{\pi(S)\pi(S')}$$
which diverges as $\pi(S)$ converges to 0 or 1. As such, when $S$ has a very small or very large probability mass, the modular bounds suffer from numerical instability. 

In the case where the global minima is near the two extremes, we observed that while the descent direction remained consistent with improvement of the objective, the instability of the bound near these extreme-mass regions caused the updates to become increasingly ill-conditioned. As the iterate approached the global minimum, the algorithm often terminated prematurely before reaching the true global minimum.


\begin{supplement}
All figures used in the paper can be found below.

\begin{figure}[h!]
    \centering
    \includegraphics[width=1\linewidth]{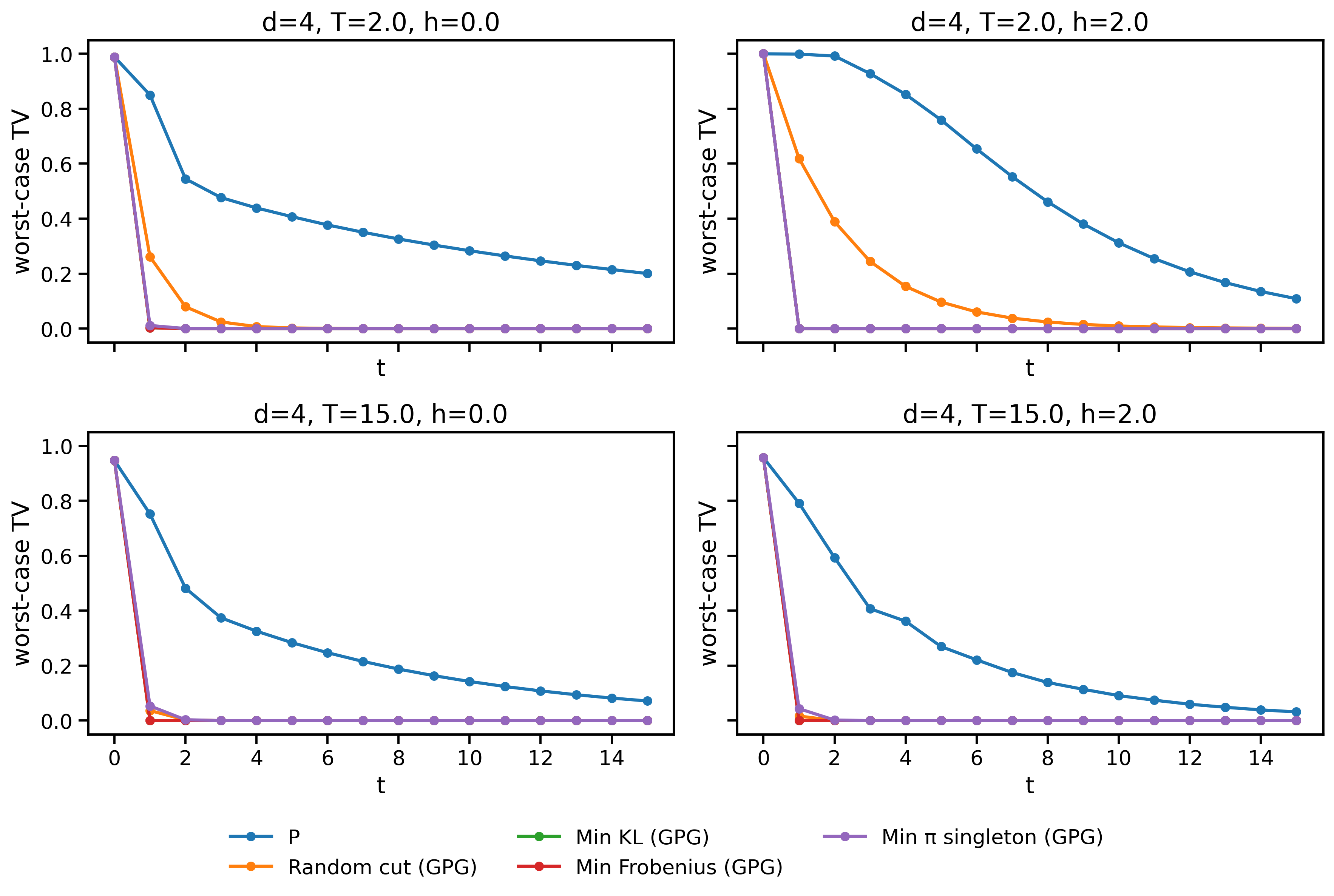}
    \caption{Plot of worst-case TV distance for $G_S P G_S$ chosen amongst different criteria}
    \label{fig:tvplot}
\end{figure}

\begin{figure}[htbp]
    \centering

    \begin{subfigure}{\linewidth}
        \centering
        \includegraphics[height=0.20\textheight]{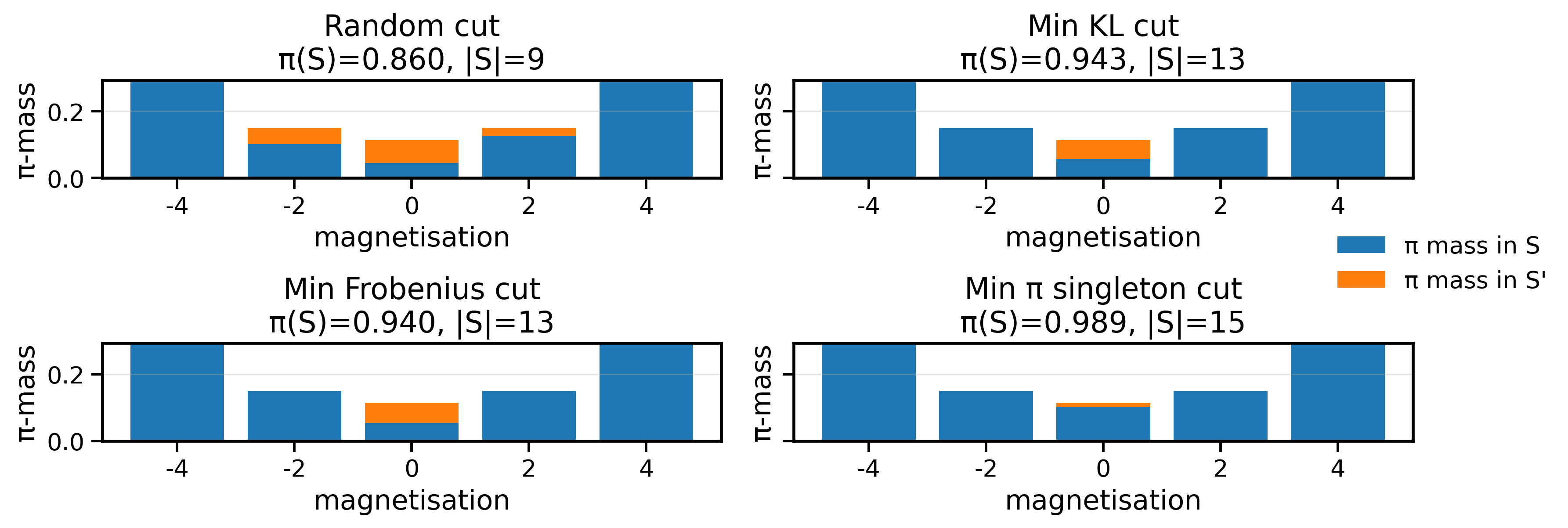}
        \caption{$T=2,\ h=0$}
    \end{subfigure}

    \begin{subfigure}{\linewidth}
        \centering
        \includegraphics[height=0.20\textheight]{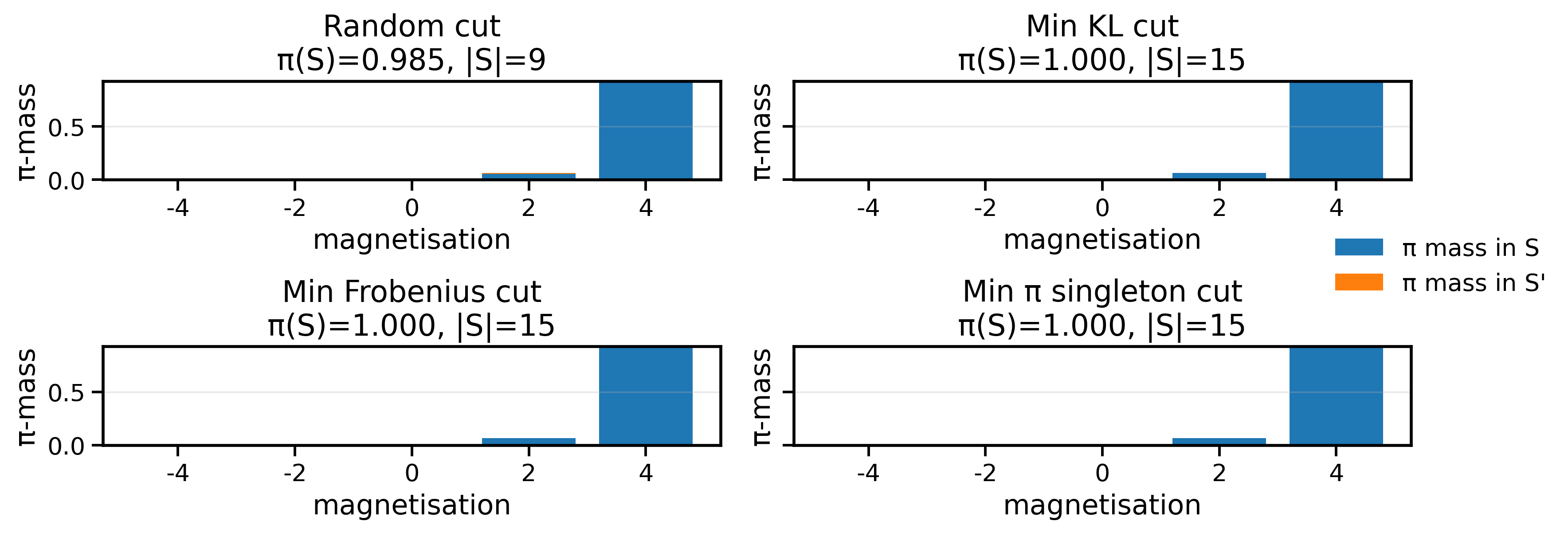}
        \caption{$T=2,\ h=2$}
    \end{subfigure}

    \begin{subfigure}{\linewidth}
        \centering
        \includegraphics[height=0.20\textheight]{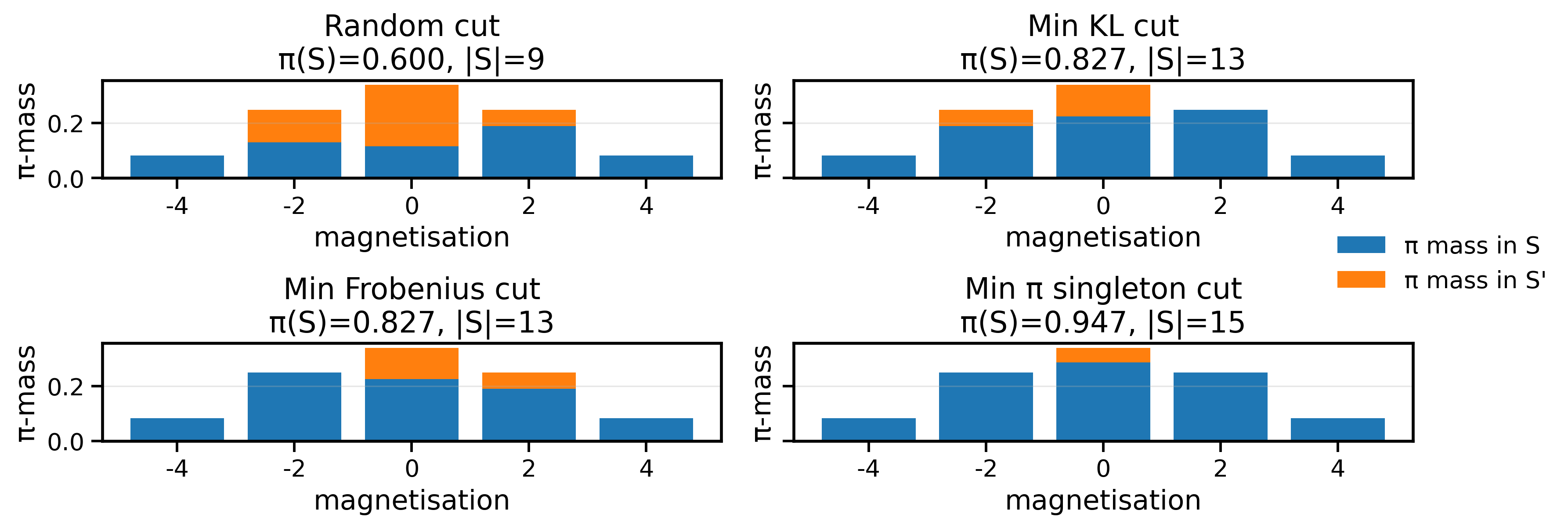}
        \caption{$T=15,\ h=0$}
    \end{subfigure}

    \begin{subfigure}{\linewidth}
        \centering
        \includegraphics[height=0.20\textheight]{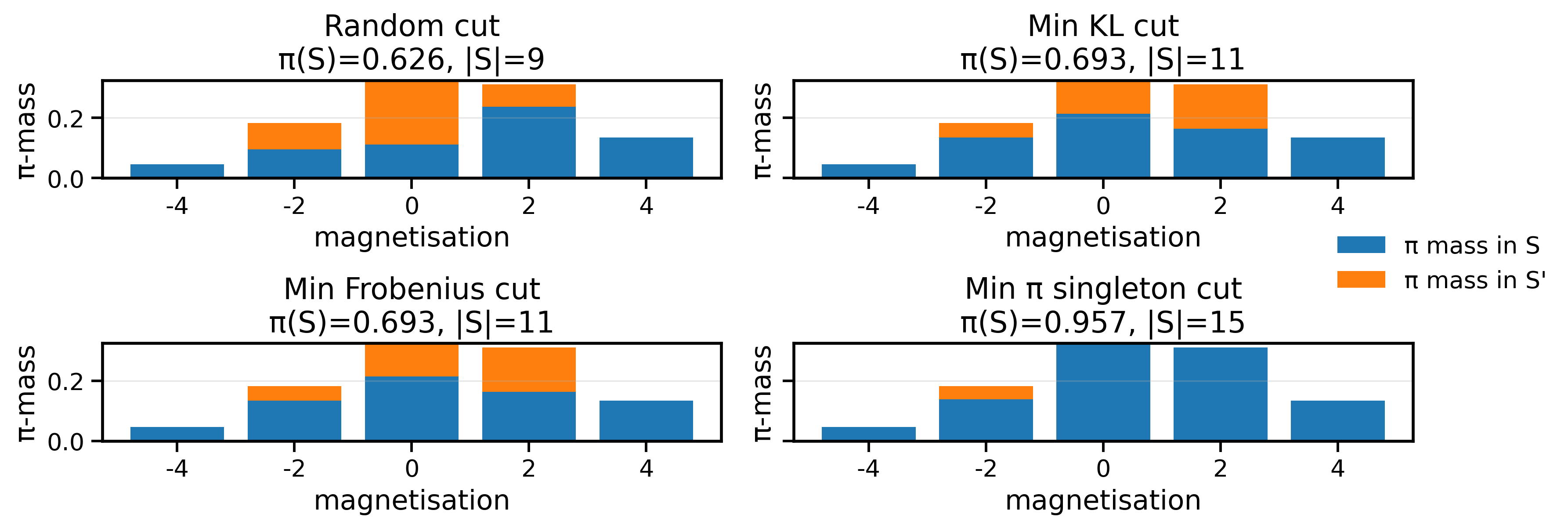}
        \caption{$T=15,\ h=2$}
    \end{subfigure}

    \caption{Cut visualisation by magnetisation for the Curie--Weiss model with $d=4$.}
    \label{fig:cut_by_magnetisation}
\end{figure}

\begin{figure}[htbp]
    \centering
    \includegraphics[width=1\linewidth]{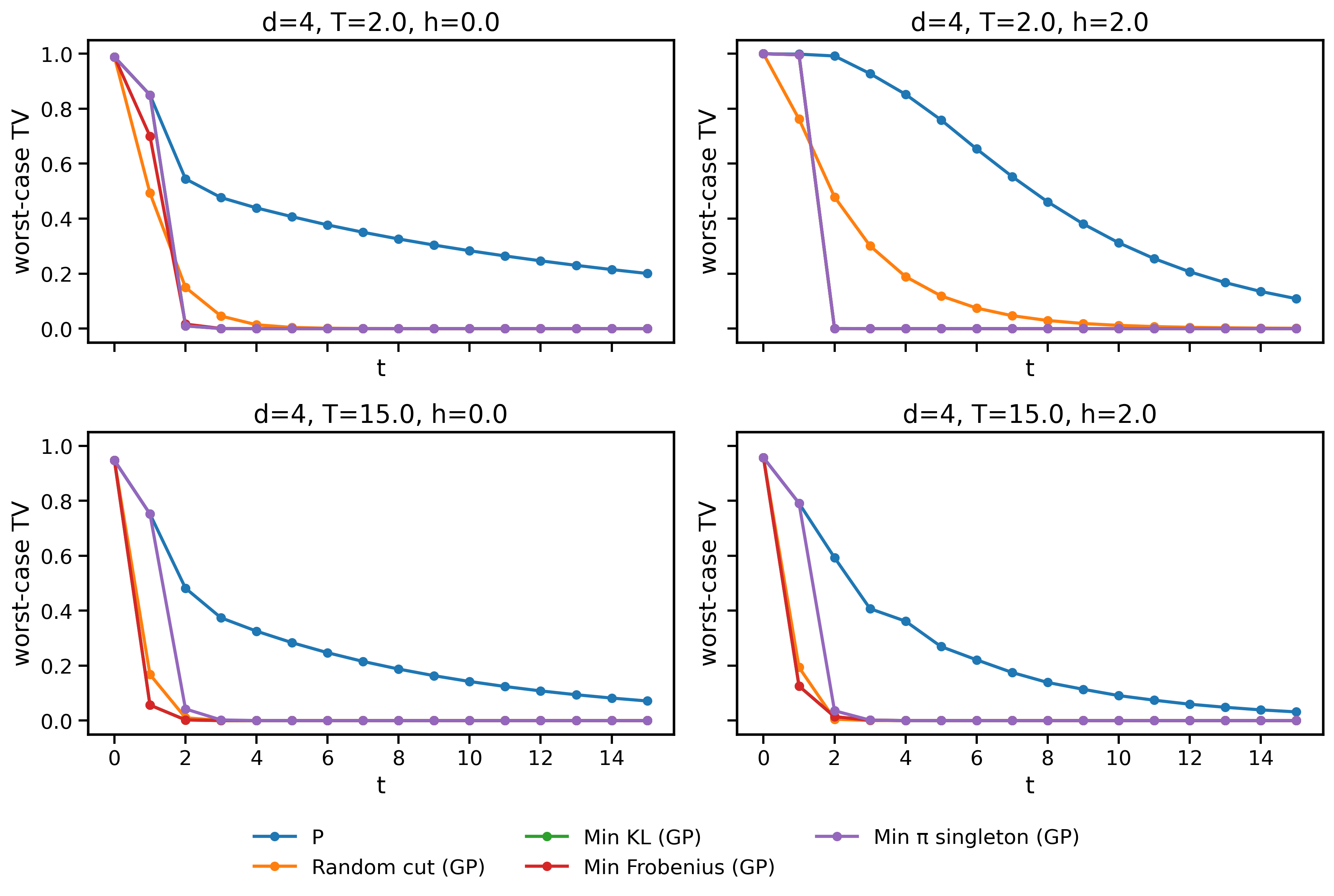}
    \caption{Plot of worst-case TV distance for $G_S P$ chosen amongst different criteria}
    \label{fig:tvplotGS}
\end{figure}

\begin{figure}[htbp]
    \centering

    \begin{subfigure}{\linewidth}
        \centering
        \includegraphics[height=0.20\textheight]{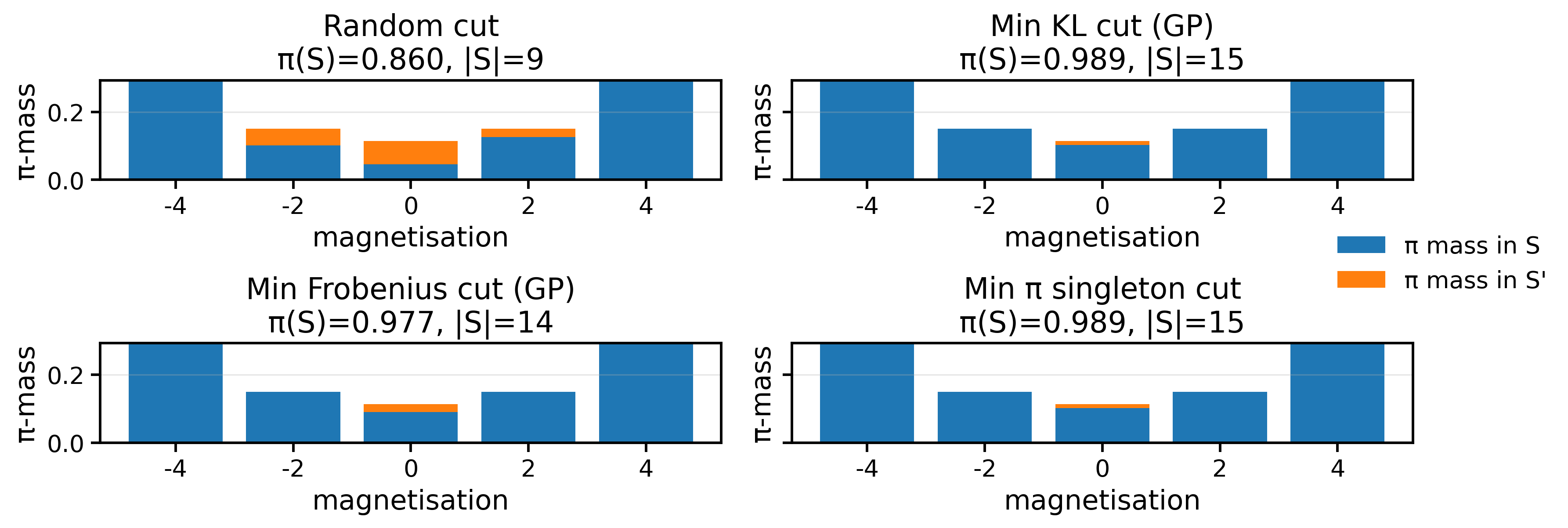}
        \caption{$T=2,\ h=0$}
    \end{subfigure}

    \begin{subfigure}{\linewidth}
        \centering
        \includegraphics[height=0.20\textheight]{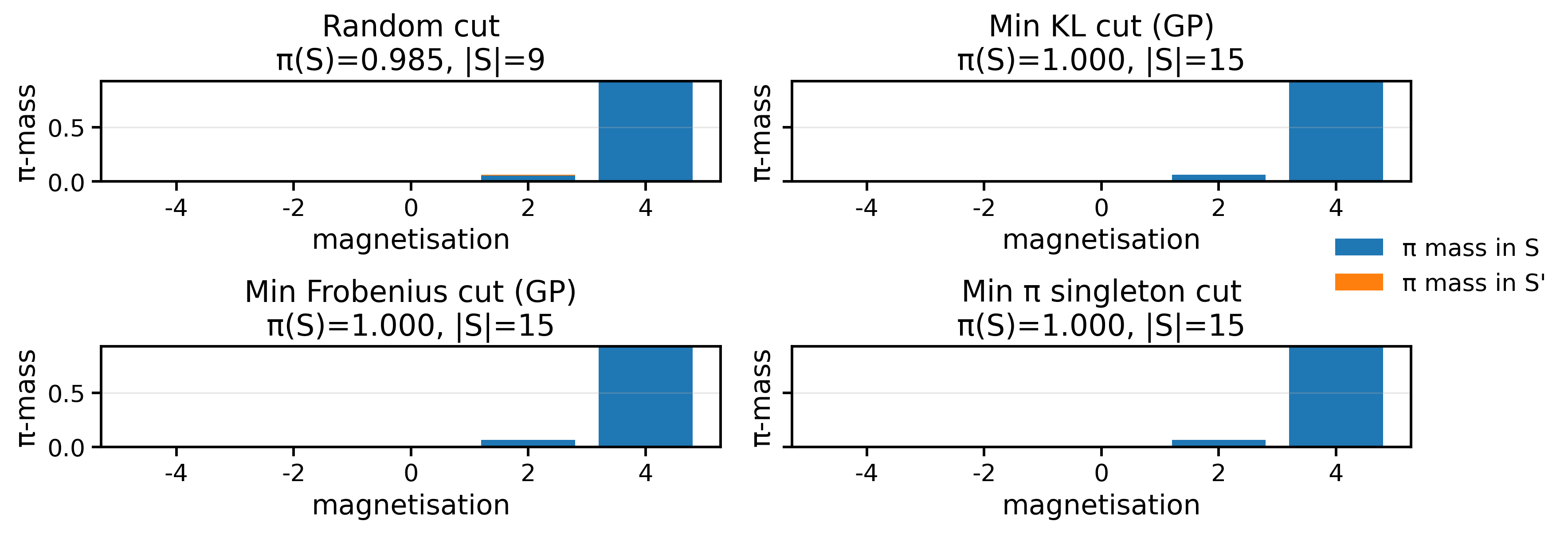}
        \caption{$T=2,\ h=2$}
    \end{subfigure}

    \begin{subfigure}{\linewidth}
        \centering
        \includegraphics[height=0.20\textheight]{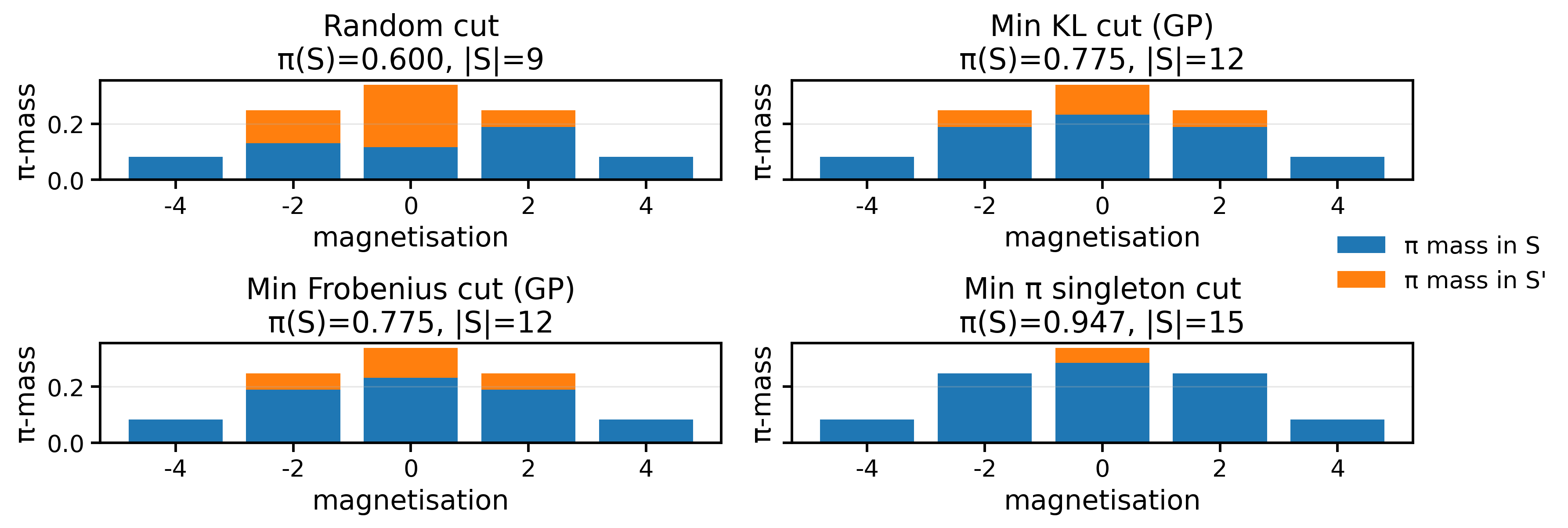}
        \caption{$T=15,\ h=0$}
    \end{subfigure}

    \begin{subfigure}{\linewidth}
        \centering
        \includegraphics[height=0.20\textheight]{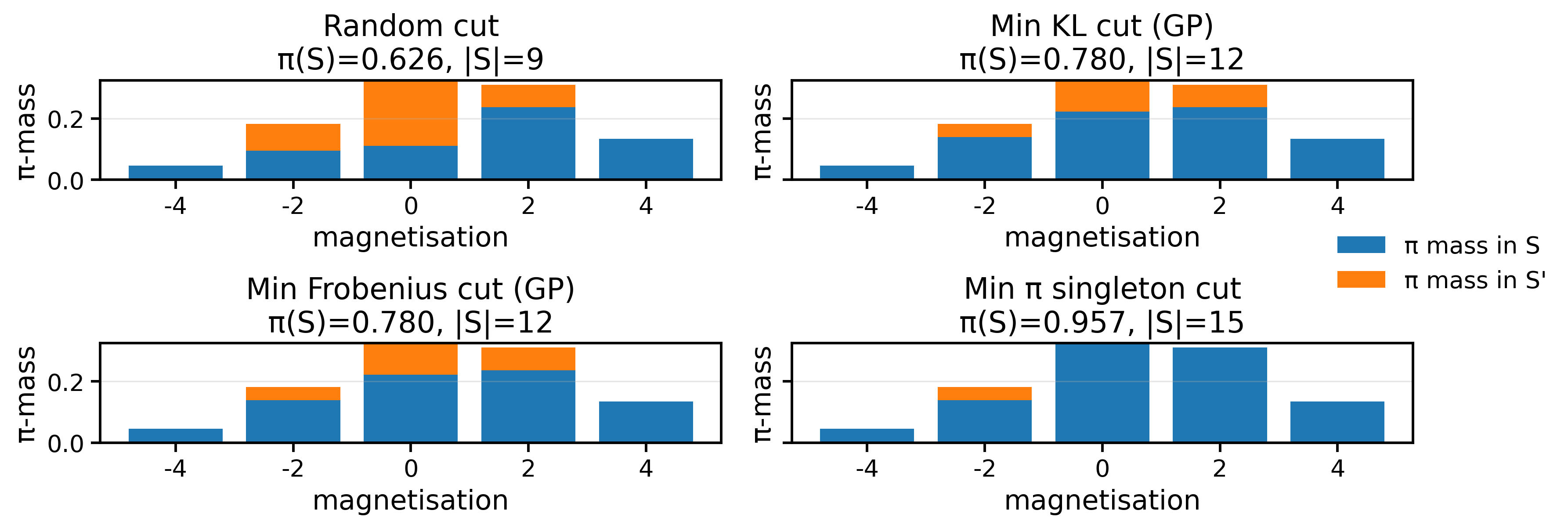}
        \caption{$T=15,\ h=2$}
    \end{subfigure}

    \caption{Cut visualisation by magnetisation for the Curie--Weiss model with $d=4$.}
    \label{fig:cut_by_magnetisationGP}
\end{figure}

\begin{figure}[htbp]
    \centering

    \begin{subfigure}{\linewidth}
        \centering
        \includegraphics[height=0.20\textheight]{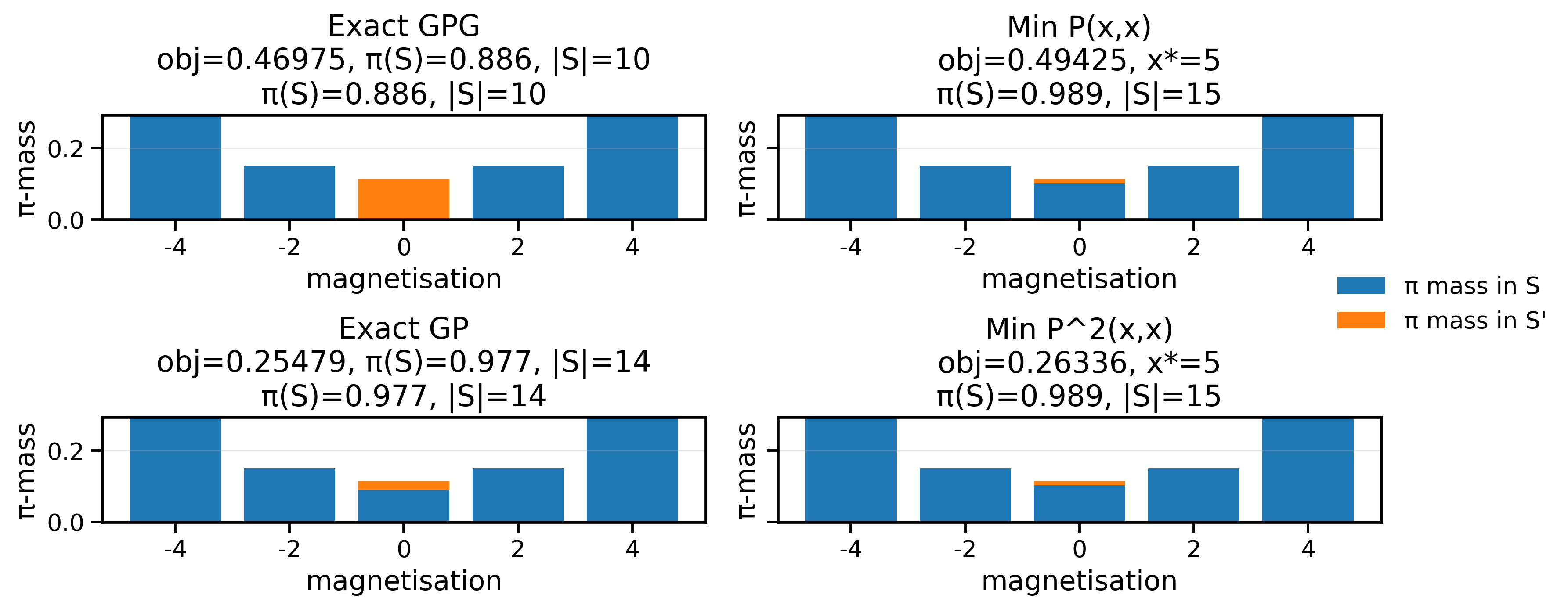}
        \caption{$T=2,\ h=0$}
    \end{subfigure}

    \begin{subfigure}{\linewidth}
        \centering
        \includegraphics[height=0.20\textheight]{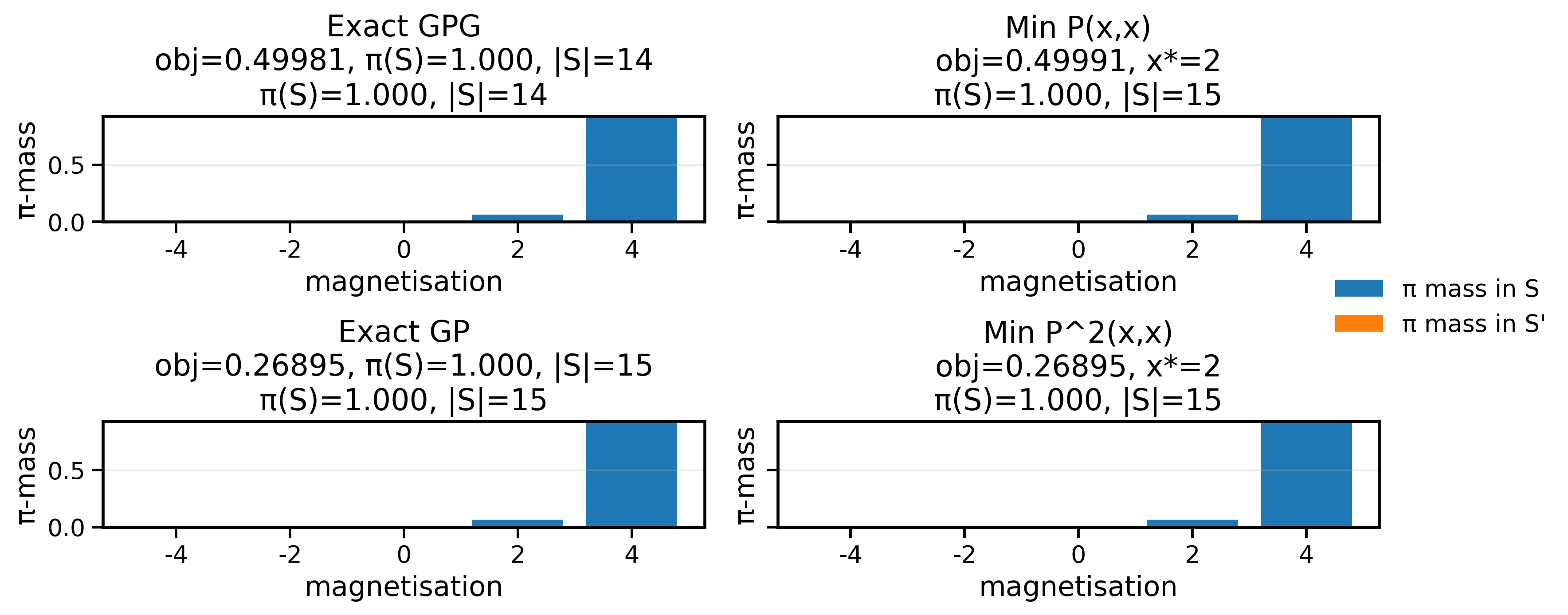}
        \caption{$T=2,\ h=2$}
    \end{subfigure}

    \begin{subfigure}{\linewidth}
        \centering
        \includegraphics[height=0.20\textheight]{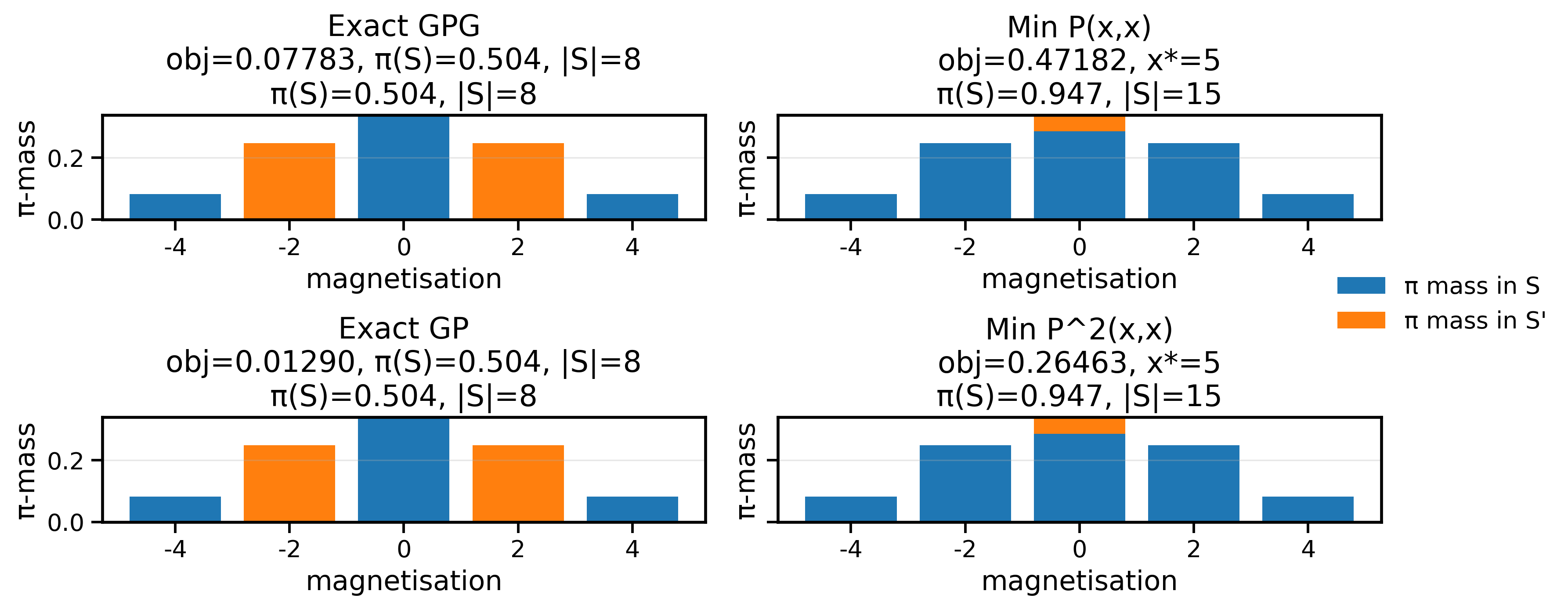}
        \caption{$T=15,\ h=0$}
    \end{subfigure}

    \begin{subfigure}{\linewidth}
        \centering
        \includegraphics[height=0.20\textheight]{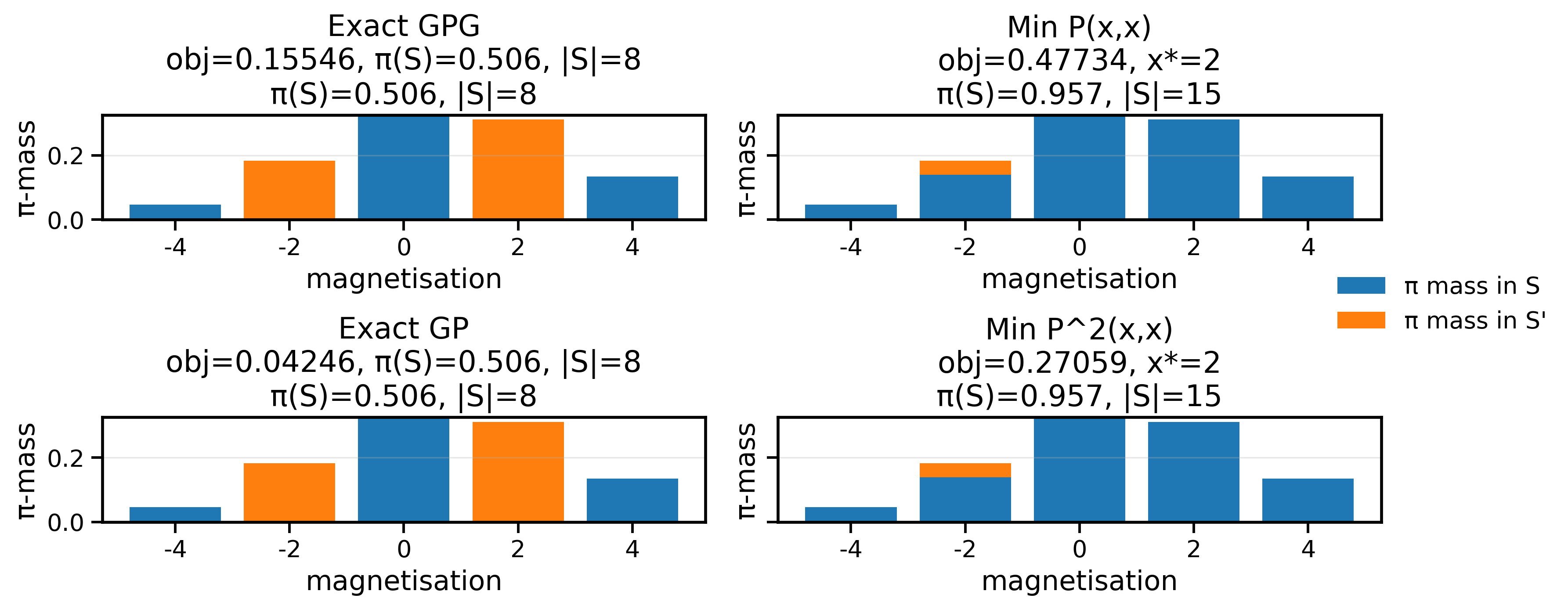}
        \caption{$T=15,\ h=2$}
    \end{subfigure}

    \caption{Comparison of cuts between true Frobenius objective via brute-force search and 1/2-approximate minimiser given in Proposition \ref{prop:1/2-approx} and \ref{prop:1/2-approxGPG}}
    \label{fig:1/2approx}
\end{figure}

\end{supplement}

\clearpage
\bibliographystyle{unsrtnat}
\bibliography{ref}

\begin{acks}
Michael Choi acknowledges the financial support of the projects A-0000178-02-00 and A-8003574-00-00 at National University of Singapore.
\end{acks}


\end{document}